\documentclass[12pt,reqno]{amsart}
\usepackage[margin=1in]{geometry}
\makeatletter
\def\section{\@startsection{section}{1}%
	\z@{.7\linespacing\@plus\linespacing}{.5\linespacing}%
	{\bfseries
		\centering
}}
\def\@secnumfont{\bfseries}
\makeatother
\usepackage{amsmath,amssymb,amsthm,graphicx,amsxtra, setspace}
\usepackage[utf8]{inputenc}
\usepackage{charter}
\usepackage{mathrsfs}
\usepackage{alltt}
\usepackage{relsize}
\usepackage{hyperref}
\usepackage{aliascnt}
\usepackage{tikz}
\usepackage{mathtools}
\usepackage{multicol}
\usepackage{upgreek}
\usepackage{graphicx,type1cm,eso-pic,color}
\allowdisplaybreaks
\usepackage{scalerel,stackengine}
\stackMath
\newcommand\reallywidehat[1]{%
	\savestack{\tmpbox}{\stretchto{%
			\scaleto{%
				\scalerel*[\widthof{\ensuremath{#1}}]{\kern-.6pt\bigwedge\kern-.6pt}%
				{\rule[-\textheight/2]{1ex}{\textheight}}
			}{\textheight}%
		}{0.5ex}}%
	\stackon[1pt]{#1}{\tmpbox}%
}
\newtheorem{theorem}{Theorem}[section]

\newaliascnt{lemma}{theorem}
\newtheorem{lemma}[lemma]{Lemma}
\aliascntresetthe{lemma} 

\newaliascnt{proposition}{theorem}

\aliascntresetthe{proposition}

\newaliascnt{assumption}{theorem}
\newtheorem{assumption}[assumption]{Assumption}
\aliascntresetthe{assumption}

\newaliascnt{corollary}{theorem}

\aliascntresetthe{corollary}

\newaliascnt{definition}{theorem}
\newtheorem{definition}[definition]{Definition}
\aliascntresetthe{definition}

\newaliascnt{example}{theorem}
\newtheorem{example}[example]{Example}
\aliascntresetthe{example}

\newaliascnt{remark}{theorem}
\newtheorem{remark}[remark]{Remark}
\aliascntresetthe{remark}

\newaliascnt{hypothesis}{theorem}

\aliascntresetthe{hypothesis}

\newaliascnt{property}{theorem}

\aliascntresetthe{property}

\let\originalleft\left
\let\originalright\right
\renewcommand{\left}{\mathopen{}\mathclose\bgroup\originalleft}
\renewcommand{\right}{\aftergroup\egroup\originalright}

\makeatletter
\newcommand{\doublewidetilde}[1]{{%
		\mathpalette\double@widetilde{#1}%
}}
\newcommand{\double@widetilde}[2]{%
	\sbox\z@{$\m@th#1\widetilde{#2}$}%
	\ht\z@=.9\ht\z@
	\widetilde{\box\z@}%
}
\makeatother

\renewcommand{\d}{\/\mathrm{d}\/}
\def\u{\textbf{X}^{n, \varepsilon}}

\def\n{\textbf{X}_{\theta_n}}

\def\w{\textbf{W}^{\varepsilon}_{{\theta}^{\varepsilon}}}

\def\L{\mathrm{L}}
\def\A{\mathrm{A}}

\def\F{\mathrm{F}}
\def\C{\mathrm{C}}

\def\h{\mathbf{h}}
\def\J{\mathrm{J}}
\def\B{\mathrm{B}}
\def\D{\mathrm{D}}
\def\y{\mathbf{y}}

\def\X{\mathbb{X}}

\def\z{\mathbf{z}}
\def\v{\mathbf{v}}
\def\V{\mathbb{V}}
\def\w{\mathbf{w}}
\def\W{\mathrm{W}}
\def\G{\mathbb{G}}

\def\no{\nonumber}
\def\V{\mathbb{V}}

\def\U{\mathrm{U}}

\def\u{\mathbf{u}}
\def\H{\mathbb{H}}
\def\n{\mathbf{n}}
\def\p{\mathbf{p}}

\newcommand{\R}{\mathbb{R}}

\renewcommand{\d}{\/\mathrm{d}\/}

\newcommand{\Addresses}{{
		\footnote{

			\noindent \textsuperscript{1}School of Mathematics,
			Indian Institute of Science Education and Research, Trivandrum (IISER-TVM),
			Maruthamala PO, Vithura, Thiruvananthapuram, Kerala, 695 551, INDIA.  \par\nopagebreak \noindent
			\textit{e-mail:} \texttt{tania9114@iisertvm.ac.in}

			\noindent \textsuperscript{2}School of Mathematics, Indian Institute of Science Education and Research, Trivandrum (IISER-TVM),
			Maruthamala PO, Vithura, Thiruvananthapuram, Kerala, 695 551, INDIA  \par\nopagebreak \noindent
			\textit{e-mail:} \texttt{sheetal@iisertvm.ac.in}

			\noindent \textsuperscript{3}Department of Mathematics, Indian Institute of Technology (IIT), Roorkee,
			Haridwar Highway, Roorkee,  Uttarakhand 247667, INDIA \par\nopagebreak
			\noindent  \textit{e-mail:} \texttt{maniltmohan@gmail.com, manil$\_$vs@isibang.ac.in}
			
			\noindent \textsuperscript{*}Corresponding author.

			\medskip\noindent
			{\bf Acknowledgments:}  Tania Biswas  would like to thank the  Indian Institute of Science Education and Research, Thiruvananthapuram, for providing financial support and stimulating environment for the research. M. T. Mohan would like to thank the Department of Science and Technology (DST), India for Innovation in Science Pursuit for Inspired Research (INSPIRE) Faculty Award (IFA17-MA110) and Indian Institute of Technology (IIT), Roorkee, for providing stimulating scientific environment and resources.
			
}}}

\begin{document}
	
	\title[Pontryagin Maximun Principle and Second order optimality conditions]{Pontryagin Maximun Principle and Second order optimality conditions for optimal control problems governed by 2D nonlocal Cahn-Hillard-Navier-Stokes equations	\Addresses	}

	\author[T. Biswas, S. Dharmatti and M. T. Mohan ]
	{Tania Biswas\textsuperscript{1}, Sheetal Dharmatti\textsuperscript{2*} and Manil T. Mohan\textsuperscript{3}} 
\maketitle

	\begin{abstract}
		In this paper, we formulate a distributed optimal control problem related to the evolution of two isothermal, incompressible, immisible fluids in a two dimensional  bounded domain. The distributed optimal control problem is framed as the minimization of a suitable cost functional subject to the controlled nonlocal Cahn-Hilliard-Navier-Stokes equations.   We describe the first order necessary conditions of optimality via Pontryagin minimum principle and prove second order necessary and sufficient conditions of optimality for the problem.
	\end{abstract}

	

\section{Introduction}\label{sec1}\setcounter{equation}{0}
 We consider the evolution of two isothermal, incompressible, immiscible fluids in a bounded domain $ \Omega\subset\R^2 $ or $ \R^3$.   The  average velocity of the fluid  is denoted by $\u(x,t)$ and the relative concentration of one fluid  is denoted by $\varphi(x,t),$ for $(x,t)\in\Omega\times(0,T)$. A general model for such a system is known as \emph{nonlocal Cahn-Hilliard-Navier-Stokes (CHNS) system}  $ \text{ in } \ \Omega\times (0,T)$  and is given by 
\begin{equation}\label{1.1}
\left\{
\begin{aligned}
\varphi_t + \u\cdot \nabla \varphi &=\text{ div} (m(\varphi) \nabla \mu), \ \text{ in } \  \Omega \times (0,T),\\
\mu &= a \varphi - \J\ast \varphi + \F'(\varphi),\\
\u_t - 2 \text{div } ( \nu(\varphi) \mathrm{D}\u ) + (\u\cdot \nabla )\u + \nabla \uppi &= \mu \nabla \varphi + \mathbf{h}, \ \text{ in } \ \Omega \times (0,T),\\
\text{div }\u&= 0, \ \text{ in } \  \Omega \times (0,T), \\
\frac{\partial \mu}{\partial\mathbf{n}} &= 0 \ , \u=0 \ \text{ on } \ \partial \Omega \times (0,T),\\
\u(0) &= \u_0, \  \ \varphi(0) = \varphi _0 \ \text{ in } \ \Omega, 
\end{aligned}   
\right.
\end{equation}
where
$m$ is the mobility parameter, $\mu$ is the \emph{chemical potential}, $\uppi$ is the \emph{pressure}, $\J$ is the \emph{spatial-dependent internal kernel}, $\J \ast \varphi$ denotes the spatial convolution over $\Omega$, $a$ is defined by $a(x) := \int _\Omega \J(x-y) \d y$, $\F$ is a double well potential, $\nu$ is the \emph{kinematic viscosity} and $\mathbf{h}$ is the external forcing term acting in the mixture.  (see \cite{unique}). Also, $\mathrm{D}\u$ is the symmetric part of the gradient of the flow velocity vector, i.e., $\mathrm{D}\u$ is the strain tensor $\frac{1}{2}\left(\nabla\u+(\nabla\u)^{\top}\right)$. The chemical potential $\mu$ is the first variation of the functional:
\begin{equation*}
\mathcal{E}(\varphi) = f(\varphi):= \frac{1}{4}\int_\Omega \int_\Omega \mathrm{J}(x-y)(\varphi(x)-\varphi(y))^2\d x  \d y + \int_\Omega \mathrm{F}(\varphi(x))\d x.
\end{equation*} The density is supposed to be constant and is equal to one (i.e., matched densities). The system (\ref{1.1}) is  called nonlocal because of the term $\J$, which is averaged over the spatial domain. Various simplified models of this system are studied by several mathematicians and physicists.   The local version of the system  is obtained by replacing  $\mu $ equation by 
$ \mu = \Delta \varphi + \F' (\varphi) $.  Another simplification appeared in the literature is to assume that the constant mobility parameter and/or constant viscosity.
From the mathematical point of view, the nonlocal version is physically more relevant and mathematically challenging too. This model is more difficult to handle because of the nonlinear terms like the \emph{capillarity term} (i.e., \emph{Korteweg force}) $\mu \nabla \varphi$ acting on the fluid.  Even in two dimensions, this term can be less regular than the convective term $(\u\cdot\nabla)\u$ (see \cite{weak}).

We now discuss some of the works available in literature for the solvability of the system (\ref{1.1}) and also the  simplified Cahn-Hilliard-Navier-Stokes models. In \cite{exist local chns}, author studies the local Cahn-Hilliard-Navier-Stokes system and establishes the existence of weak solutions in dimensions 2 and 3. The existence and uniqueness of strong solutions in 2 and 3 dimensions  (global in 2D and local in time for 3D) is also established in \cite{exist local chns}. The authors in \cite{weak} proved the existence of a weak solution for nonlocal Cahn-Hilliard-Navier-Stokes system with mobility parameter  equal to one and variable viscosity. The uniqueness of weak solution for such systems remained open until  2016 and the authors in \cite{unique} resolved it for dimension 2. The authors in  \cite{unique} also considered the case of nonlocal systems with variable mobility parameter and viscosity coefficient under certain assumptions on the kernel $\J$.  The existence of a unique strong solution in two dimensions for the nonlocal system  with constant viscosity and mobility parameter equal to 1 is proved in \cite{strong} and the authors showed that any weak solution regularises in finite time uniformly with respect to bounded sets of initial data.  As in the case of 3D Navier-Stokes, in three dimensions, the existence of a weak solution is known (see \cite{weak}), but the uniqueness of weak solution  for the nonlocal Cahn-Hilliad-Navier-Stokes system still remains open. 

Optimal  control theory of fluid dynamic models has been one of the far-reaching areas of applied mathematics with  several engineering  applications (see for example \cite{sritharan,fursikov,gunzburger}). Controlling fluid flow and turbulence inside a flow in a given physical domain, with known initial data and by various means, for example, body forces, boundary values, temperature (cf. \cite{FART,raymond2} etc), is an interesting problem in fluid mechanics. The mathematical developments in infinite dimensional nonlinear system theory and partial differential equations in the past several decades, opened up a new window for the optimal control theory of Navier-Stokes equations. Such problems are extensively addressed  in \cite{sritharan,fursikov,gunzburger,FART,lions,raymond,SDMTS} etc. 

The optimal control problem for the Cahn-Hilliard system \cite{optimal control for ch eq,optimal control for ch eq with state constraints} and optimal control problem for Cahn-Hilliard system with dynamic boundary control \cite{optimal1} are available in the literature. In \cite{analysis,double}, authors discuss about the control problems related to the phase field system of Cahn-Hilliard type. Optimal control of time discrete two phase flow is studied in \cite{num3} and topological optimization for phase field models is studied in \cite{num2}. Turning to the local Cahn-Hilliard-Navier-Stokes  equations, an optimal control problem with state constraint and robust control are investigated in \cite{robust,state const}, respectively. An optimal distributed control of a diffuse interface model of tumor growth is considered in \cite{tumor}. The model studied in \cite{tumor} is  a kind of local Cahn-Hilliard type system. Optimal control problems of semi discrete Cahn-Hilliard-Navier-Stokes system for various cases like distributed and boundary control, with non smooth Ginzburg-Landau energies and with non matched fluid densities are studied in \cite{num4,num5,num1}. These works considered the local Cahn-Hilliard-Navier-Stokes equations for their numerical studies. 

The first order necessary conditions of optimality for various optimal control problems governed by nonlocal Cahn-Hilliad-Navier-Stokes system has been established in \cite{ControlCHNS,SFMG,BDM2}, etc. In \cite{BDM2}, the authors have studied a distributed optimal control problem for nonlocal Cahn-Hilliard-Navier-Stokes system and established Pontryagin’s maximum principle using Ekeland's Variational Principle. They have also studied a initial value optimization problem. An optimal distributed control problem for the two-dimensional nonlocal Cahn-Hilliard-Navier-Stokes systems with degenerate mobility and singular potential is studied in \cite{SFMG}. A distributed optimal control problem for the nonlocal Cahn-Hilliard-Navier-Stokes system with non-constant viscosity and regular potential  is examined in \cite{ControlCHNS}. 
The second order optimality condition for 3D Navier-Stokes equations in a  periodic domain is established in \cite{LijuanPezije}. The second order optimality condition for various optimization problems governed by Navier-Stokes equations is obtained in \cite{CT,TW,WD}, etc.

In this paper, we are studying a distributed optimal control problem related to \eqref{1.1} with constant viscosity and mobility parameter $m=1$. We give a systematic approach to the mathematical formulation of the optimal control problem  and resolve the problem of minimizing total energy.  We consider two dimensional fluid flows, since the nonlocal Cahn-Hilliard-Navier-Stokes equations are not known to be well-posed in three dimensions.  We  establish Pontryagin maximum principle for the controlled nonlocal Cahn-Hilliard-Navier-Stokes system. The unique global strong solution of the system (\ref{nonlin phi})-(\ref{initial conditions}) (see below) established in \cite{strong} helps us to achieve this goal. Then we prove second order necessary and sufficient conditions of optimality for the optimal control problem. The coupling in the system (\ref{1.1}) makes the problem mathematically challenging and harder to resolve than that of the corresponding problem for the Navier-Stokes and Cahn-Hilliard systems.

The paper is organized as follows: In the next section, we discuss the functional setting for the unique  solvability of the system (\ref{nonlin phi})-(\ref{initial conditions}) (see below). We also state the existence and uniqueness of a weak as well as strong solution of the system in the same section.  Then we state the unique solvability of the  linearized system, which we have proved in \cite{BDM2}. In the section \ref{se4}, an optimal control problem is formulated as the minimization of a suitable cost functional (the sum of total energy and total effort by controls). We first prove the existence of an optimal control (see Theorem \ref{optimal}) and then establish  the Pontryagin maximum principle (see Theorem \ref{main}), which gives the first-order necessary optimality conditions for the associated optimal control problem. We characterize the optimal control using the adjoint system. In the section \ref{se5}, we have obtained the second order necessary and sufficient optimality condition for the optimal control problem (see Theorems \ref{necessary} and \ref{sufficient}).

\section{Mathematical Formulation}\label{sec2}
 In this section, we mathematically formulate the two dimensional Cahn-Hilliard-Navier-Stokes system and discuss the necessary function spaces required to obtain the global solvability results for such systems. We mainly follow the papers \cite{unique,weak} for the mathematical formulation and functional setting. 
\subsection{\bf Governing Equations}
A well known model which describes the evolution of an incompressible isothermal mixture of two immiscible fluids is governed by  \emph{Cahn-Hilliard-Navier-Stokes system} (see \cite{weak}).  
We consider the following controlled Cahn-Hilliard-Navier-Stokes system:
\begin{subequations}
	\begin{align}
	\varphi_t + \u\cdot \nabla \varphi &= \Delta \mu, \label{nonlin phi}\ \text{ in }\ \Omega\times(0,T),\\
	\mu &= a \varphi - \J\ast \varphi + \F'(\varphi), \label{mu}\ \text{ in }\ \Omega\times(0,T),\\
	\u_t - \nu \Delta \u + (\u\cdot \nabla )\u + \nabla \uppi &= \mu \nabla \varphi + \mathbf{h} +\U, \label{nonlin u}\ \text{ in }\ \Omega\times(0,T),\\
	\text{div }\u&= 0, \ \text{ in }\ \Omega\times(0,T), \label{div zero}\\
	\frac{\partial \mu}{\partial\mathbf{n}} &= 0 \ , \u=0 \ \text{ on } \ \partial \Omega \times (0,T),\label{boundary conditions}\\
	\u(0) &= \u_0, \  \ \varphi(0) = \varphi _0 \ \text{ in } \ \Omega, \label{initial conditions}
	\end{align}   
\end{subequations}
where $\Omega \subset \mathbb{R}^2$ is a bounded domain with sufficiently smooth boundary and $\mathbf{n}$ is the unit outward normal to the boundary $\partial\Omega$. In the system (\ref{nonlin phi})-(\ref{initial conditions}),  $\U$ is the \emph{distributed control} acting on the system.

\subsection{\bf Functional Setting}
Let us introduce the following functional spaces required for getting the unique global solvability results of the system (\ref{nonlin phi})-(\ref{initial conditions}).
\begin{align*}
\G_{\text{div}} &:= \Big\{ \u \in \mathrm{L}^2(\Omega;\R^2) : \text{div }\u=0,\  \u\cdot \mathbf{n}\big|_{\partial\Omega}=0 \Big\}, \\
\V_{\text{div}} &:=\Big\{\u \in \mathrm{H}^1_0(\Omega;\R^2): \text{div }\u=0\Big\},\\ \mathrm{H} &:=\mathrm{L}^2(\Omega;\R),\ \mathrm{V}:=\mathrm{H}^1(\Omega;\R). \end{align*}
Let us denote $\| \cdot \|$ and $(\cdot, \cdot)$ the norm and the scalar product, respectively, on both $\mathrm{H}$ and $\G_{\text{div}}$. The duality between any Hilbert space $\X$ and its dual $\X'$ will be denoted by $\left<\cdot,\cdot\right>$. We know that $\V_{\text{div}}$ is endowed with the scalar product 
$$(\u,\v)_{\V_{\text{div}} }= (\nabla \u, \nabla \v)=2(\mathrm{D}\u,\mathrm{D}\v),\ \text{ for all }\ \u,\v\in\V_{\text{div}}.$$ The norm on $\V_{\text{div}}$ is given by $\|\u\|_{\V_{\text{div}}}^2:=\int_{\Omega}|\nabla\u(x)|^2\d x=\|\nabla\u\|^2$. Since $\Omega$ is bounded, the embedding of $\mathbb{V}_{\text{div}}\subset\G_{\text{div}}\equiv\G_{\text{div}}'\subset\V_{\text{div}}'$ is compact (see \cite{Te}). In the sequel, we use the notations $\mathbb{H}^2(\Omega):=\mathrm{H}^2(\Omega;\mathbb{R}^2)$ and  $\mathrm{H}^2(\Omega):=\mathrm{H}^2(\Omega;\mathbb{R})$ for second order Sobolev spaces. 
\subsection{\bf Linear and Nonlinear Operators}
Let us define the Stokes operator $\A : \D(\A)\cap  \G_{\text{div}} \to \G_{\text{div}}$ by 
$$\A=-\mathrm{P} \Delta,\ \D(\A)=\mathbb{H}^2(\Omega) \cap \V_{\text{div}},$$ where $\mathrm{P} : \mathbb{L}^2(\Omega) \to \G_{\text{div}}$ is the \emph{Helmholtz-Hodge orthogonal projection}. Note also that, we have
\begin{equation*}
\langle\A\u, \v\rangle = (\u, \v)_{\V_\text{div}} = (\nabla\u, \nabla\v),  \text{ for all } \u \in\D(\A), \v \in \V_{\text{div}}.
\end{equation*}
It should also be noted that  $\A^{-1} : \G_{\text{div}} \to \G_{\text{div}}$ is a self-adjoint compact operator on $\G_{\text{div}}$ and by
the classical \emph{spectral theorem}, there exists a sequence $\lambda_j$ with $0<\lambda_1\leq \lambda_2\leq \lambda_j\leq\cdots\to+\infty$
and a family  $\mathbf{e}_j \in \D(\A)$ of eigenvectors is orthonormal in $\G_\text{div}$ and is such that $\A\mathbf{e}_j =\lambda_j\mathbf{e}_j$. We know that $\u$ can be expressed as $\u=\sum\limits_{j=1}^{\infty}\langle \u,\mathbf{e}_j\rangle \mathbf{e}_j,$ so that $\A\u=\sum\limits_{j=1}^{\infty}\lambda_j\langle \u,\mathbf{e}_j\rangle \mathbf{e}_j$. Thus, it is immediate that 
\begin{equation}
\|\nabla\u\|^2=\langle \A\u,\u\rangle =\sum_{j=1}^{\infty}\lambda_j|\langle \u,\mathbf{e}_j\rangle|^2\geq \lambda_1\sum_{j=1}^{\infty}|\langle \u,\mathbf{e}_j\rangle|^2=\lambda_1\|\u\|^2,
\end{equation} 
which is the \emph{Poincar\'e inequality}.
For $\u,\v,\w \in \V_{\text{div}}$ we define the trilinear operator $b(\cdot,\cdot,\cdot)$ as
$$b(\u,\v,\w) = \int_\Omega (\u(x) \cdot \nabla)\v(x) \cdot \w(x)\d x=\sum_{i,j=1}^2\int_{\Omega}u_i(x)\frac{\partial v_j(x)}{\partial x_i}w_j(x)\d x,$$
and the bilinear operator $\B$ from $\V_{\text{div}} \times \V_{\text{div}} $ into $\V_{\text{div}}'$ defined by,
$$ \langle \B(\u,\v),\w  \rangle := b(\u,\v,\w), \  \text{ for all } \ \u,\v,\w \in \V_\text{{div}}.$$
An integration by parts yields, 
\begin{equation*}
\left\{
\begin{aligned}
b(\u,\v,\v) &= 0, \ \text{ for all } \ \u,\v \in\V_\text{{div}},\\
b(\u,\v,\w) &=  -b(\u,\w,\v), \ \text{ for all } \ \u,\v,\w\in \V_\text{{div}}.
\end{aligned}
\right.\end{equation*}
For more details about the linear and nonlinear operators, we refer the readers to \cite{Te}. Now, we give some important inequalities which are used in the rest of the paper. 
\begin{lemma}[Gagliardo-Nirenberg Inequality, Theorem 2.1, \cite{ED}] \label{gn}
	Let $\Omega\subset\R^n$ and $\u\in\mathrm{W}^{1,p}(\Omega;\R^n),p\geq 1$. Then for any fixed number $p,q\geq 1$, there exists a constant $C>0$ depending only on $n,p,q$ such that 
	\begin{equation}\label{2.3}
	\|\u\|_{\mathbb{L}^r}\leq C\|\nabla\u\|_{\mathbb{L}^p}^{\theta}\|\u\|_{\mathbb{L}^q}^{1-\theta},\;\theta\in[0,1],
	\end{equation}
	where the numbers $p, q, r$ and $\theta$ satisfy the relation
	$$\theta=\left(\frac{1}{q}-\frac{1}{r}\right)\left(\frac{1}{n}-\frac{1}{p}+\frac{1}{q}\right)^{-1}.$$
\end{lemma}
A particular case of Lemma \ref{gn} is the well known inequality due to Ladyzhenskaya (see Lemma 1 and 2, Chapter 1, \cite{OAL}), which is true even in unbounded domains and, is given below:
\begin{lemma}[Ladyzhenskaya Inequality]\label{lady}
	For $\u\in\ \C_0^{\infty}(\Omega;\R^n), n = 2, 3$, there exists a constant $C$ such that
	\begin{equation}
	\|\u\|_{\mathbb{L}^4}\leq C^{1/4}\|\u\|^{1-\frac{n}{4}}\|\nabla\u\|^{\frac{n}{4}},\text{ for } n=2,3,
	\end{equation}
	where $C=2,4$ for $n=2,3$ respectively. 
\end{lemma}
Note that the constant appearing in Ladyzhenskaya inequality does not depend on $\Omega$. Thus, for every $\u,\v,\w \in \V_{\text{div}}$, the following estimates hold:
\begin{equation}
|b(\u,\v,\w)| \leq \sqrt{2}\|\u\|^{1/2}\| \nabla \u\|^{1/2}\|\v\|^{1/2}\| \nabla \v\|^{1/2}\| \nabla \w\|,
\end{equation}
so that for all $\u\in\V_{\text{div}},$ we have
\begin{equation}
\label{be}
\|\B(\u,\u)\|_{\V_{\text{div}}'}\leq \sqrt{2}\|\u\|\|\nabla\u\|\leq \sqrt{\frac{2}{\lambda_1}}\|\u\|_{\V_{\text{div}}}^2 ,
\end{equation}
by using the Poincar\'e inequality. 

Taking $n=p=q=2$ and in \eqref{2.3}, we get $\theta=1-\frac{2}{r}$, so that for all $r\geq 2$, we have 
\begin{equation}
\|\u\|_{\mathbb{L}^r}\leq C\|\nabla\u\|^{1-\frac{2}{r}}\|\u\|^{\frac{2}{r}},
\end{equation}

\begin{lemma}[Agmon's Inequality, Lemma 13.2, \cite{SA}]
	For any 
	$\u\in \H^{s_2}(\Omega;\R^n),$ choose  $s_{1}$ and $s_{2}$ such that 
	$ s_1< \frac{n}{2} < s_2$. Then, if 
	$0< \alpha < 1$ and 
	$\frac{n}{2} = \alpha s_1 + (1-\alpha)s_2$, the following inequality holds 
	$$ \|\u\|_{\mathbb{L}^\infty}\leq C \|\u\|_{\H^{s_1}}^{\alpha} \|\u\|_{\H^{s_2}}^{1-\alpha}.$$
\end{lemma}
For $\u\in\H^2(\Omega)\cap\H_0^1(\Omega)$, the \emph{Agmon's inequality} in 2D states that there exists a constant 
$C>0$ such that
\begin{equation}\label{agm}
\|\u\|_{\mathbb{L}^{\infty}}\leq C\|\u\|^{1/2}\|\u\|_{\H^2}^{1/2}\leq C\|\u\|_{\H^2}.
\end{equation}
For every $f \in \mathrm{V}'$ we denote $\overline{f}$ the average of $f$ over $\Omega$, i.e., $\overline{f} := |\Omega|^{-1}\langle f, 1 \rangle$, where $|\Omega|$ is the Lebesgue measure of $\Omega$. 
Let us also introduce the spaces (see \cite{unique})
\begin{align*}\mathrm{V}_0 &= \{ v \in \mathrm{V} \ : \ \overline{v} = 0 \},\\
\mathrm{V}_0' &= \{ f \in \mathrm{V}' \ : \ \overline{f} = 0 \},
\end{align*}
and the operator $\mathcal{A} : \mathrm{V} \rightarrow \mathrm{V}'$ is defined by
\begin{equation*}\langle \mathcal{A} u ,v \rangle := \int_\Omega \nabla u(x) \cdot \nabla v(x) \d x \  \text{for all } \ u,v \in \mathrm{V}.
\end{equation*}
Clearly $\mathcal{A}$ is linear and it maps $\mathrm{V}$ into $\mathrm{V}_0'$ and its restriction $\mathcal{B}$ to $\mathrm{V}_0$ onto $\mathrm{V}_0'$ is an isomorphism.   
We know that for every $f \in \mathrm{V}_0'$, $\mathcal{B}^{-1}f$ is the unique solution with zero mean value of the \emph{Neumann problem}:
\begin{equation*}
\left\{
\begin{aligned}
- \Delta u = f, \  \mbox{ in } \ \Omega, \\
\frac{\partial u}{\partial\mathbf{n}} = 0, \ \mbox{ on } \  \partial \Omega.
\end{aligned}
\right.\end{equation*}
In addition, we have
\begin{eqnarray} 
&\langle \mathcal{A}u , \mathcal{B}^{-1}f \rangle = \langle f ,u \rangle, \ \text{ for all } \ u\in \mathrm{V},  \ f \in \mathrm{V}_0' , \label{bes}\\
&\langle f , \mathcal{B}^{-1}g \rangle = \langle g ,\mathcal{B}^{-1}f \rangle = \int_\Omega \nabla(\mathcal{B}^{-1}f)\cdot \nabla(\mathcal{B}^{-1}g)\d x, \ \text{for all } \ f,g \in \mathrm{V}_0'.\label{bes1}
\end{eqnarray}
Note that $\mathcal{B}$ can be also viewed as an unbounded linear operator on $\mathrm{H}$ with
domain $\D(\mathcal{B}) = \left\{v \in \mathrm{H}^2(\Omega) : \frac{\partial v}{\partial\mathbf{n}}= 0\text{ on }\partial\Omega \right\}$. 
\begin{definition}
Let $\X$ be a Banach space and
$\tau\in\R$, we denote by $\mathrm{L}^p_{\text{tb}} ([\tau,\infty) ;\X)$, $1\leq  p<\infty$, the space of functions 
$f \in \mathrm{L}^p_{\text{loc}} ([\tau,\infty); \X)$ that are translation-bounded in $\mathrm{L}_{\text{loc}}^p ([\tau,\infty); \X)$, that is, $$\|f\|^p_{\mathrm{L}^p_{\text{tb}} ([\tau,\infty) ;\X)}:=\sup_{t\geq \tau}\int_t^{t+1}\|f(s)\|_{\X}^p\d s<+\infty.$$
\end{definition}
\subsection{\bf Existence and Uniqueness of the Governing Equations}
Now we state the existence theorem and uniqueness theorem for the nonlocal Cahn-Hilliard-Navier-Stokes system  given in \eqref{nonlin phi}-\eqref{initial conditions}.  Let us first make the following assumptions:
\begin{assumption}\label{prop of F and J} Let
	$\mathrm{J}$ and $\mathrm{F}$ satisfy:
	\begin{enumerate}
		\item [(1)] $ \J \in \W^{1,1}(\mathbb{R}^2;\R), \  \J(x)= \J(-x) \; \text {and} \ a(x) = \int_\Omega \J(x-y)\d y \geq 0,$ a.e., in $\Omega$.
		\item [(2)] $\F \in \C^{2}(\mathbb{R})$ and there exists $C_0 >0$ such that $\F''(s)+ a(x) \geq C_0$, for all $s \in \mathbb{R}$, a.e., $x \in \Omega$.
		\item [(3)] Moreover, there exist $C_1 >0$, $C_2 > 0$ and $q>0$ such that $\F''(s)+ a(x) \geq C_1|s|^{2q} - C_2$, for all $s \in \mathbb{R}$, a.e., $x \in \Omega$.
		\item [(4)] There exist $C_3 >0$, $C_4 \geq 0$ and $r \in (1,2]$ such that $|\F'(s)|^r \leq C_3|\F(s)| + C_4,$ for all $s \in \mathbb{R}$.
	\end{enumerate}
\end{assumption}
\begin{remark}\label{remark J}
	Assumption $\J \in \W^{1,1}(\mathbb{R}^2;\R)$ can be weakened. Indeed, it can be replaced by $\J \in \W^{1,1}(\B_\delta;\R)$, where $\B_\delta := \{z \in \mathbb{R}^2 : |z| < \delta \}$ with $\delta := \emph{diam}(\Omega)=\sup\limits_{x,y\in \Omega}d(x,y)$, where $d(\cdot,\cdot)$ is the Euclidean metric on $\mathbb{R}^2$, or also by
	 \begin{equation} \label{Estimate J}
	\sup_{x\in \Omega} \int_\Omega \left( |\J(x-y)| + |\nabla \J(x-y)| \right) \d y < +\infty .
	\end{equation}
\end{remark}

\begin{remark}\label{remark F}
	Since $\F(\cdot)$ is bounded from below, it is easy to see that the Assumption \ref{prop of F and J} (4) implies that $\F(\cdot)$ has a polynomial growth of order $r'$, where $r' \in [2,\infty)$ is the conjugate index to $r$. Namely, there exist $C_5$ and $C_6 \geq 0$ such that 
	\begin{eqnarray}\label{2.9}
	|\F(s)| \leq C_5|s|^{r'} + C_6,\   \text{for all } \ s \in \R.
	\end{eqnarray}
	Observe that the Assumption \ref{prop of F and J} (4) is fulfilled by a potential of arbitrary polynomial growth. For example,  the Assumption \ref{prop of F and J}  (2)-(4) are satisfied for the case of the well-known double-well potential $$\F (s)= (s^2 - 1)^2.$$
\end{remark}
Let us now give the definition of a weak solution for the system  \eqref{nonlin phi}-\eqref{initial conditions}.
\begin{definition}[Weak Solution]
	Let $\u_0\in\G_{\text{div}}$, $\varphi_0\in \mathrm{H}$ with $\F(\varphi_0)\in\mathrm{L}^1(\Omega)$ and $0<T<\infty$ be given. Then $(\u,\varphi)$ is said to be a \emph{weak solution} to the uncontrolled system  \eqref{nonlin phi}-\eqref{initial conditions} on $[0,T]$ corresponding to the initial conditions $\u_0$ and $\varphi_0$ if 
	\begin{itemize}
		\item [(i)] $\u,\varphi$ and $\mu$ satisfy 
		\begin{equation}\label{sol}
		\left\{
		\begin{aligned}
		&	\u \in \mathrm{L}^{\infty}(0,T;\G_{\text{div}}) \cap \mathrm{L}^2(0,T;\V_{\text{div}}),  \\ 
		&	\u_t \in \mathrm{L}^{2-\gamma}(0,T;\V_{\text{div}}'),\text{ for all  }\gamma \in (0,1),  \\
		&	\varphi \in \mathrm{L}^{\infty}(0,T;\mathrm{H}) \cap \mathrm{L}^2(0,T;\mathrm{V}),   \\
		&	\varphi_t \in \mathrm{L}^{2-\delta}(0,T;\mathrm{V}'),\text{ for all  }\delta \in (0,1), \\
		&	\mu \in \mathrm{L}^2(0,T;\mathrm{V}),
		\end{aligned}
		\right.
		\end{equation}
		\item [(ii)]  For every $\psi\in\mathrm{V}$, every $\v \in \V_{\text{div}}$, if we define 		             $\rho$ by 
		\begin{equation}
		\rho(x,\varphi):=a(x)\varphi+\F'(\varphi),
		\end{equation} and for almost any $t\in(0,T)$, we have
		\begin{align}
		\langle \varphi_t,\psi\rangle +(\nabla\rho,\nabla\psi)&=\int_{\Omega}(\u\cdot\nabla\psi)\varphi\d x+\int_{\Omega}(\nabla\J*\varphi)\cdot\nabla\psi\d x,\\
		\langle \u_t,\v\rangle +\nu(\nabla\u,\nabla\v)+b(\u,\v,\w)&=-\int_{\Omega}(\v\cdot\nabla\mu)\varphi\ d x+\langle\h,\v\rangle.
		\end{align}
		\item [(iii)] Moreover, the following initial conditions hold in the weak sense 
		\begin{equation}
		\u(0)=\u_0,\ \varphi(0)=\varphi_{0},
		\end{equation}
		i.e., for every $\v \in \V_{\text{div}}$, we have $(\u(t),\v) \to (\u_0,\v)$ as $t\to 0$, and for every $\chi \in \mathrm{V}$, we have $(\varphi(t),\chi) \to (\varphi_0,\chi)$ as $t\to  0$.
	\end{itemize}
\end{definition}
Next, we discuss the existence and uniqueness of weak solution results available in the literature for  the system  \eqref{nonlin phi}-\eqref{initial conditions}.
\begin{theorem}[Existence, Theorem 1, Corollaries 1 and 2, \cite{weak}]\label{exist}
	Let the Assumption \ref{prop of F and J} be satisfied. Let $\u_0 \in \G_{\text{div}}$, $\varphi_0 \in \mathrm{H}$ such that $\F(\varphi_0) \in \mathrm{L}^1(\Omega)$ and $\mathbf{h} \in \mathrm{L}^2_{\text{loc}}([0,\infty), \V_{\text{div}}')$. Then, for every given $T>0$, there exists a weak solution $(\u,\varphi)$ to the uncontrolled equation \eqref{nonlin phi}-\eqref{initial conditions} such that \eqref{sol} is satisfied.
	Furthermore, setting
	\begin{equation*}\mathscr{E}(\u(t),\varphi(t)) = \frac{1}{2} \|\u(t)\|^2 + \frac{1}{4} \int_\Omega \int_\Omega \J(x-y) (\varphi(x,t) - \varphi(y,t))^2 \d x \d y + \int_\Omega \F(\varphi(t))\d x,
\end{equation*}
	the following energy estimate holds for almost any $t>0$:
	\begin{equation}\label{energy}
	\mathscr{E}(\u(t),\varphi(t)) + \int_0^t \left(\nu \| \nabla \u(s)\|^2 + \| \nabla\mu(s) \|^2 \right)\d s \leq \mathscr{E}(\u_0,\varphi_0) + \int_0^t \langle \mathbf{h}(s), \u(s) \rangle\d s,\end{equation}
	or the weak solution $(\u,\varphi)$ satisfies the following energy identity,
	$$\frac{\d}{\d t}\mathscr{E}(\u(t),\varphi(t)) + \nu \|\nabla \u(t) \|^2+ \| \nabla \mu(t) \|^2 = \langle \mathbf{h}(t) , \u(t) \rangle.$$
	
	Furthermore, if in addition $\mathbf{h}  \in \mathrm{L}^2_{\text{tb}}([0,\infty);\V_{\text{div}}')$, then the following dissipative estimate is satisfied:
	$$\mathscr{E}(\u(t),\varphi(t)) \leq \mathscr{E}(\u_0,\varphi_0) \exp (-kt) + \F(m_0)|\Omega| + K,  $$
	where $m_0 = (\varphi_0,1)$ and $k$, $K$ are two positive constants which are independent of the initial data, with $K$ depending on $\Omega$, $\nu$, $\J$, $\F$, $\|\mathbf{h}\|_{\mathrm{L}_{\text{tb}}^2(0,\infty;\V_{\text{div}}')}$. 
\end{theorem} 
\begin{remark}
	The above theorem also implies $\u\in\C([0,T];\G_{\text{div}})$ and $\varphi\in \C([0,T];\mathrm{H})$ by using the Aubin-Lions compactness theorem.
\end{remark}

\begin{remark}\label{rem2.5}
	We denote by $\mathbb{Q}$, a continuous monotone increasing function with respect to each of its arguments. As a consequence of energy inequality (\ref{energy}), we have the following bound:
	\begin{eqnarray}
	&	\|\u\|_{\mathrm{L}^{\infty}(0,T;\G_{\text{div}}) \cap \mathrm{L}^2(0,T;\V_{\text{div}})} + \|\varphi \|_{\mathrm{L}^{\infty}(0,T;\mathrm{H}) \cap \mathrm{L}^2(0,T;\mathrm{V})} + \|\F(\varphi)\|_{\mathrm{L}^{\infty}(0,T;\mathrm{H})} \nonumber\\& \leq \mathbb{Q}\left(\mathscr{E}(\u_0,\varphi_0),\|\mathbf{h}\|_{\mathrm{L}^2(0,T;\V_{\text{div}}')}\right),
	\end{eqnarray}
	where $\mathbb{Q}$ also depends on $\F$, $\J$, $\nu$ and $\Omega$.
\end{remark}

\begin{theorem}[Uniqueness, Theorem 2, \cite{unique}]\label{unique}
	Suppose that the Assumption \ref{prop of F and J} is satisfied. Let $\u_0 \in \G_{\text{div}}$, $\varphi_0 \in \mathrm{H}$ with $\F(\varphi_0) \in \mathrm{L}^1(\Omega)$ and $\mathbf{h} \in \mathrm{L}^2_{\text{loc}}([0,\infty);\V_{\text{div}}')$. Then, the weak solution $(\u,\varphi)$ corresponding to $(\u_0,\varphi_0)$ and given by Theorem \ref{exist} is unique. Furthermore, for $i=1,2$, let $\z_i :=  (\u_i,\varphi_i)$ be two weak solutions corresponding to two initial data $\z_{0i} :=  (\u_{0i},\varphi_{0i})$ and external forces $\h_i$, with $\u_{0i} \in \G_{\text{div}}$, $\varphi_{0i} \in \mathrm{H}$ with $\F(\varphi_{0i}) \in \mathrm{L}^1(\Omega)$ and $\h_i \in \mathrm{L}^2_{\text{loc}}([0,\infty);\V_{\text{div}}')$. Then the following continuous dependence estimate holds:
	\begin{align*}
	&\| \u_2(t) - \u_1(t) \|^2 + \| \varphi_2(t) - \varphi_1(t) \|^2_{\mathrm{V}'} \\
	&\quad + \int_0^t \left( \frac{C_0}{2} \| \varphi_2(\tau) - \varphi_1(\tau) \|^2 + \frac{\nu}{4} \|\nabla( \u_2(\tau) - \u_1(\tau)) \|^2 \right) d \tau \\
	&\leq \left( \|\u_2(0) - \u_1(0) \|^2 + \| \varphi_2(0) - \varphi_1(0) \|^2_{\mathrm{V}'} \right) \Lambda_0(t) \\
	&\quad + \| \overline{\varphi}_2(0) - \overline{\varphi}_1(0)\| \mathbb{Q} \left( \mathcal{E}(z_{01}),\mathcal{E}(z_{02}),\| \h_1 \|_{\mathrm{L}^2(0,t;\V_{div}')},\| \h_2 \|_{\mathrm{L}^2(0,t;\V_{div}')} \right)  \Lambda_1(t) \\
	&\quad + \| \h_2 - \h_1 \|^2_{\mathrm{L}^2(0,T;\V_{\text{div}}')} \Lambda_2(t) ,
	\end{align*} 
	for all $t \in [0,T]$, where $\Lambda_0(t)$, $\Lambda_1(t)$ and $\Lambda_2(t)$ are continuous functions which depend on the norms of the two solutions. The functions $\mathbb{Q}$ and $\Lambda_i(t)$ also depend on $\F$, $\J$ and $\Omega$.
\end{theorem}
The following theorem gives the existence and uniqueness of strong solution for the uncontrolled system  \eqref{nonlin phi}-\eqref{initial conditions}. 
\begin{theorem}[Global Strong Solution, Theorem 2, \cite{strong}]\label{strongsol}
	Let $\h\in \mathrm{L}^2_{\text{loc}}([0,\infty);\G_{\text{div}})$, $\u_0\in\V_{\text{div}}$, $\varphi_0\in\mathrm{V}\cap\mathrm{L}^{\infty}(\Omega)$ be given and the Assumption \ref{prop of F and J} be satisfied. Then, for a given $T > 0$, there exists \emph{a unique weak solution} $(\u,\varphi)$ of  \eqref{nonlin phi}-\eqref{initial conditions}  such that
	\begin{equation}\label{0.22}
	\left\{
	\begin{aligned}
	& \u\in \mathrm{L}^{\infty}(0,T;\V_{\text{div}})\cap \mathrm{L}^2(0,T;\H^2(\Omega)), \ \ \varphi\in \mathrm{L}^{\infty}((0,T) \times\Omega)\cap \mathrm{L}^{\infty}(0,T;\mathrm{V}),\\ &\u_t\in\mathrm{L}^2(0,T;\G_{\text{div}}), \ \ \varphi_t\in\mathrm{L}^2(0,T;\mathrm{H}).
	\end{aligned}
	\right.
	\end{equation}
	Furthermore, suppose in addition that $\F \in\C^3(\mathbb{R}),\ a \in\mathrm{H}^2(\Omega)$ and that $\varphi_0\in \mathrm{H}^2(\Omega)$. Then, the system \eqref{nonlin phi}-\eqref{initial conditions}  admits \emph{a unique strong solution} on $[0, T ]$ satisfying (\ref{0.22}) and also
	\begin{equation}
	\label{0.23}\left\{
	\begin{aligned}
	&\varphi \in\mathrm{L}^{\infty}(0,T;\mathrm{W}^{1,p}),\ 2\leq p<\infty,\\ &\varphi_t\in \mathrm{L}^{\infty}(0,T;\mathrm{H})\cap\mathrm{L}^2(0,T;\mathrm{V}).
	\end{aligned}
	\right.
	\end{equation}
	If $\J \in\mathrm{W}^{2,1}(\mathbb{R}^2;\R)$, we have in addition
	\begin{equation}\label{0.24}\varphi\in \mathrm{L}^{\infty}(0, T ; \mathrm{H}^2(\Omega)). \end{equation}
\end{theorem}

\begin{remark}\label{rem2.12}
	The regularity properties given in (\ref{0.22})-(\ref{0.24}) imply that
	\begin{equation}\label{ues}
	\u\in\C  ([0, T] ; \V_{\text{div}}), \ \varphi \in \C ( [0, T ]; \mathrm{V} )  \cap \C_w ( [0, T ]; \mathrm{H}^ 2(\Omega)).
	\end{equation}
	Moreover we also have strong continuity in time, that is,	
	\begin{equation}
	\varphi\in\C([0,T];\mathrm{H}^2(\Omega)).
	\end{equation}
\end{remark}
The following weak-strong uniqueness result (see Theorem 6, \cite{unique}), Theorem \ref{thm2.16} (see Lemma 2.6, \cite{ControlCHNS}) and Remark \ref{imp} (see below) are used in section \ref{se4}. We can show the following two results by relaxing the condition on $\J$ (see Definition 2, \cite{unique}), since the CHNS system under our consideration is having constant viscosity.
\begin{theorem}[Weak-Strong Uniqueness, Theorem 6, \cite{unique}]\label{weakstrong}
	Let the Assumption \ref{prop of F and J} be satisfied. Let $\u_0\in \mathbb{G}_{\text{div}}$, $\varphi_0\in\mathrm{V}\cap\mathrm{L}^{\infty}(\Omega)$ and let $(\u_1,\varphi_1)$ be a weak solution and $(\u_2,\varphi_2)$ a strong solution satisfying \eqref{0.22}, both corresponding to $(\u_0,\varphi_0)$ and to the same external force $\mathbf{h}\in\mathrm{ L}^ 2 ( 0 , T ; \G_{\text{div}} )$. Then, $\u_1 = \u_2$ and $\varphi_1=\varphi_2$, satisfying the following differential inequality:
\begin{equation}
\frac{1}{2}\frac{\d}{\d t}\left(\|\u\|^2+\|\varphi\|^2\right)+\frac{\nu}{2}\|\nabla\u\|^2+\frac{C_0}{4}\|\nabla\varphi\|^2\leq \Pi\left(\|\u\|^2+\|\varphi\|^2\right),
\end{equation}
where $\u=\u_1-\u_2$, $\varphi=\varphi_1-\varphi_2$ and the function $\Pi$ is given by 
\begin{equation*}
\Pi=C\left(1+\|\nabla\u_2\|^2\|\u_2\|_{\mathbb{H}^2}^2+\|\nabla\u_1\|^2+\|\varphi_1\|_{\mathrm{L}^4}^2+\|\varphi_2\|_{\mathrm{L}^4}^2+\|\nabla\varphi_2\|_{\mathbb{L}^4}^2+\|\nabla\varphi_2\|_{\mathbb{L}^4}^4\right).
\end{equation*}
\end{theorem}
\begin{theorem}[Lemma 2.6, \cite{ControlCHNS}]\label{thm2.16}
Let the Assumption \ref{prop of F and J} be satisfied. Let $\u_0\in \mathbb{G}_{\text{div}}$, $\varphi_0\in\mathrm{V}\cap\mathrm{L}^{\infty}(\Omega)$ and let $(\u_1,\varphi_1)$  and $(\u_2,\varphi_2)$ be two strong solutions satisfying \eqref{0.22}, both corresponding to $(\u_0,\varphi_0)$ and to controls  $\mathrm{U}_1,\mathrm{U}_2\in\mathrm{ L}^ 2 ( 0 , T ; \G_{\text{div}} )$, respectively. Then,  we have 
\begin{align}
&\|\u_1-\u_2\|^2_{\C([0,T];\G_{\text{div}})}+\|\u_1-\u_2\|^2_{\mathrm{L}^{2}(0,T;\V_{\text{div}})}+\|\varphi_1-\varphi_2\|_{\C([0,T];\mathrm{V})}^2\nonumber\\&\quad +\|\varphi_1-\varphi_2\|_{\mathrm{L}^2(0,T;\mathrm{H}^2)}^2+\|\varphi_1-\varphi_2\|_{\mathrm{H}^1(0,T;\mathrm{H})}^2\nonumber\\&\leq \mathbb{Q}_1\left(\|\U_1\|_{\mathrm{L}^2(0,T;\G_{\text{div}})},\|\U_2\|_{\mathrm{L}^2(0,T;\G_{\text{div}})}\right)\|\mathrm{U}_1-\U_2\|_{\mathrm{L}^2(0,T;\G_{\text{div}})}^2
\end{align}
where  $\mathbb{Q}_1 : [0, \infty)^2\to[0,\infty)$ is nondecreasing in both its arguments and only depends on the data $\F, \J, \nu, \Omega, T, \u_0$, and $\varphi_0$.
\end{theorem}

\begin{remark}\label{imp}
	If $(\u_1,\varphi_1)$ and $(\u_2,\varphi_2)$ are two strong solutions of the system \eqref{nonlin phi}-\eqref{initial conditions} corresponding to the controls $\U_1$ and $\U_2$ respectively. Then , from the Theorem \ref{thm2.16}, we also infer that 
	\begin{equation*}
	\|\varphi_1-\varphi_2\|_{\mathrm{L}^p(0,T;\mathrm{V})}\leq C\|\U_1-\U_2\|_{\mathrm{L}^2(0,T;\G_{\text{div}})}, \; \; \text{for all} \; \; p > 1
	\end{equation*} 
\end{remark}

\subsection{Existence and Uniqueness of the Linearized System}\label{se3}
Our goal in this section is to establish the existence of an optimal control for the system \eqref{nonlin phi}-\eqref{initial conditions} with an appropriate cost functional. From the well known theory for the optimal control problems  governed by partial differential equations, we know that the optimal control is derived in terms of the adjoint variable which satisfies a linear system. As a first step towards our goal, we linearize the nonlinear system and obtain the existence and uniqueness of weak solution for the linearized system using a Faedo-Galerkin approximation technique.  

Let us linearize the equations \eqref{nonlin phi}-\eqref{initial conditions}  around $(\widehat{\u}, \widehat{\varphi})$ which is the \emph{unique weak solution} of system (\ref{nonlin phi})-(\ref{initial conditions}) with control term $\U=0$ (uncontrolled system), external forcing $\widehat{\h}$, and initial datum $\widehat{\u}_0$ and $\widehat{\varphi}_0$ are such that
$$\widehat{\h}\in \mathrm{L}^2(0,T;\G_{\text{div}}),\ \widehat{\u}_0\in\V_{\text{div}},\ \widehat{\varphi}_0\in\mathrm{V}\cap\mathrm{L}^{\infty}(\Omega).$$ Thus, using the Theorem  \ref{strongsol}, we know that $(\widehat{\u}, \widehat{\varphi})$ is also a strong solution of the system (\ref{nonlin phi})-(\ref{initial conditions}).

Let us now rewrite the equation \eqref{nonlin u}. We know that 
\begin{align}\label{mh} \mu\nabla \varphi = (a\varphi - \J\ast \varphi + \F'(\varphi)) \nabla \varphi= \nabla \left(\F(\varphi) + a\frac{\varphi^2}{2}\right) - \nabla a\frac{\varphi^2}{2} - (\J\ast \varphi)\nabla \varphi. \end{align}
Hence we can rewrite \eqref{nonlin u} as
\begin{equation}\label{nonlin u rewritten}
\u_t - \nu \Delta \u + (\u\cdot \nabla )\u + \nabla \widetilde{\uppi}_{\u} = - \nabla a\frac{\varphi^2}{2} - (\J\ast \varphi)\nabla \varphi +\h + \U,
\end{equation} 
where $\widetilde{\uppi}_{\u} = \uppi -\left( \F(\varphi) + a\frac{\varphi^2}{2}\right)$. Since the pressure is also an unknown quantity, in order to linearize, we substitute $\u= \w+ \widehat{\u},$ $\uppi=\widetilde\uppi+\widehat{\uppi}$ and $\varphi = \psi + \widehat{\varphi}$ in  \eqref{nonlin u rewritten} and \eqref{nonlin phi} to  get
$$\w_t - \nu \Delta \w + (\w \cdot \nabla)\widehat{\u} + (\widehat{\u} \cdot \nabla )\w + \nabla \widetilde{\uppi}_\w = -\nabla a\psi \widehat{\varphi} -(\J\ast \psi) \nabla \widehat{\varphi} - (\J \ast \widehat{\varphi}) \nabla \psi +\widetilde \h +\U, $$ 
where $ \widetilde{\uppi}_\w = \widetilde{\uppi}-(\F'(\widehat{\varphi})+a\widehat{\varphi})\psi$, $\widetilde{\h}=\h-\widehat{\h}$. Also, we have 
$$\psi_t + \w\cdot \nabla \widehat{\varphi} + \widehat{\u} \cdot \nabla \psi = \Delta \widetilde{\mu} $$   
where $\widetilde{\mu} = a \psi - \J\ast \psi + \F''( \widehat{\varphi}) \psi$.
Hence, we consider the following linearized system:
	\begin{align}
	\w_t - \nu \Delta \w + (\w \cdot \nabla )\widehat{\u} + (\widehat{\u} \cdot \nabla )\w + \nabla \widetilde{\uppi}_\w &= -\nabla a\psi \widehat{\varphi} -(\J\ast \psi) \nabla \widehat{\varphi} \nonumber \\ 
&\qquad	- (\J \ast \widehat{\varphi}) \nabla \psi + \widetilde\h +\U, \label{lin w} \\
	\psi_t + \w\cdot \nabla \widehat{\varphi} + \widehat{\u} \cdot\nabla \psi &= \Delta \widetilde{\mu}, \label{lin psi} \\ 
	\widetilde{\mu} &= a \psi - \J\ast \psi + \F''(\widehat{\varphi})\psi, \label{lin mu}\\
	\text{div }\w &= 0, \label{lin div zero}\\
	\frac{\partial \widetilde{\mu}}{\partial\mathbf{n}} &= 0, \ \w=0  \ \text{on } \ \partial \Omega \times (0,T), \label{lin boundary conditions}\\
	\w(0) &= \w_0, \ \psi(0) = \psi _0 \ \text{ in } \ \Omega. \label{lin initial conditions}
	\end{align} 
Note that in \eqref{lin mu}, we used Taylor's formula:
$$\F'(\psi+\widehat{\varphi})=\F'(\widehat{\varphi})+\F''(\widehat{\varphi})\psi+\F'''(\widehat{\varphi}+\theta\psi)\frac{\psi^2}{2},$$ for $0<\theta<1$, and ignored the second order terms in $\psi$, since we are considering a linear system, and in particular $\widehat{\varphi}\in\mathrm{L}^{\infty}((0,T)\times\Omega)$ and $\F(\cdot)$ has a polynomial growth as discussed in Remark \ref{remark F}. 
 
Next, we discuss the unique global solvability results for the system (\ref{lin w})-(\ref{lin initial conditions}).

\begin{theorem}[Existence and Uniqueness of Linearized System]\label{linearized}
	Suppose that the Assumption \ref{prop of F and J} is satisfied. Let us assume $(\widehat{\u},\widehat{\varphi})$ is the unique strong solution of the system \eqref{nonlin phi}-\eqref{initial conditions} with the regularity given in (\ref{0.22}) and $\F \in\C^3(\mathbb{R}),\ a \in\mathrm{H}^2(\Omega).$ Let $\w_0 \in \G_{\text{div}}$ and $\psi_0 \in \mathrm{H}$ with $\widetilde{\h},\U \in \mathrm{L}^2(0,T;\G_{\text{div}})$.  Then, for a given $T >0$, there exists \emph{a unique weak solution} $(\w,\psi)$ to the system \eqref{lin w}-\eqref{lin initial conditions} such that
	$$\w \in \mathrm{L}^{\infty}(0,T;\G_{\text{div}}) \cap \mathrm{L}^{2}(0,T;\V_{\text{div}})\ \text{ and 
	} \ \psi \in \mathrm{L}^{\infty}(0,T;\mathrm{H}) \cap  \mathrm{L}^{2}(0,T;\mathrm{V}).$$ and for every $\v \in \V_{\text{div}}$ and $\xi \in \mathrm{V}$ and for all $t \in (0,T)$, we have  
\begin{eqnarray}\label{weak_lin}
\left\{
	\begin{aligned}
	 \langle\w_t,\v\rangle + \nu (\nabla\w,\nabla \v) + b(\w,\widehat{\u},\v) &+ b(\widehat{\u},\w,\v)  = - (\nabla a\psi \widehat{\varphi},\v) - ((\J\ast \psi) \nabla \widehat{\varphi},\v)  \\  
&	\quad - ((\J \ast \widehat{\varphi}) \nabla \psi,\v) + (\widetilde\h,\v) + (\U,\v),  \\ 
	\langle\psi_t,\xi\rangle + (\w\cdot \nabla \widehat{\varphi},\xi) + (\widehat{\u} \cdot \nabla \psi,\xi) &= -(\nabla \widetilde{\mu},\nabla \xi), 
	\end{aligned}
	\right.
	\end{eqnarray}
	where $\w(0) = \w_0, \ \psi(0) = \psi _0$ are satisfied in the weak sense. 
\end{theorem}

\begin{proof}
Using Galerkin approximation technique, the proof easily follows. [For more details refer Theorem 3.4 of \cite{BDM2}].
\end{proof}

 \section{Optimal Control Problem}\label{se4}
In this section, we formulate a distributed optimal control problem  as the minimization of a suitable cost functional subject to the controlled nonlocal Cahn-Hilliard-Navier-Stokes system. The main aim  is to establish the existence of an optimal control that minimizes the cost functional given below, subject to the constraint  \eqref{nonlin phi}-\eqref{initial conditions}.
The associated cost functional is defined by
\begin{align}\label{cost}
\mathcal{J}(\u,\varphi,\U) := \frac{1}{2} \int_0^T \|\u(t)-\u_d(t)\|^2 \d t&+ \frac{1}{2} \int_0^T \|\varphi(t) -\varphi_d(t)\|^2 \d t\no\\&\qquad+ \frac{1}{2} \int_0^T\|\U(t)\|^2\d t,
\end{align}
where $\u_d(\cdot)\in\mathrm{L}^2(0,T;\V_{\text{div}})$ and $\varphi_d(\cdot)\in\mathrm{L}^2(0,T;\mathrm{V})$ are the desired states. Note that the cost functional is the sum of total energy and total effort by control. 

Let us assume that 
\begin{align}\label{fes}
\F \in\C^5(\mathbb{R}),\ a \in\mathrm{H}^2(\Omega),
\end{align}
and the initial data  
\begin{align}\label{initial}
\u_0\in\V_{\text{div}}\text{ and }\varphi_0\in \mathrm{H}^2(\Omega).
\end{align} By the embedding of $\mathrm{H}^{2}(\Omega)$ in $\mathrm{V}$ and $\mathrm{L}^{\infty}(\Omega)$, the initial concentration $\varphi_0\in \mathrm{H}^2(\Omega)$ implies $\varphi_0\in \mathrm{V}\cap\mathrm{L}^{\infty}(\Omega)$. From now onwards, we assume that along with the Assumption \ref{prop of F and J} condition (\ref{fes}) holds true, so that the system \eqref{nonlin phi}-\eqref{initial conditions} has a unique strong solution. 

We consider set of admissible controls to be the space consisting of controls $\U \in\mathrm{L}^{2}(0,T;\G_{\text{div}})$, and denote it by $\mathscr{U}_{\text{ad}}$.

\begin{definition}[Admissible Class]\label{definition 1}
	The \emph{admissible class} $\mathscr{A}_{\text{ad}}$ of triples $(\u,\varphi,\U)$ is defined as the set of states $(\u,\varphi)$ with initial data satisfies (\ref{initial}), solving the system \eqref{nonlin phi}-\eqref{initial conditions} with control $\U \in \mathscr{U}_{ad}.$ That is,
	\begin{align*}
	\mathscr{A}_{\text{ad}}:=\Big\{(\u,\varphi,\U) :(\u,\varphi)\text{ is \text{a unique strong solution} of }\eqref{nonlin phi}\text{-}\eqref{initial conditions}  \text{ with control }\U\Big\}.
	\end{align*}
\end{definition}
Clearly $\mathscr{U}_{\text{ad}}$ is nonempty and note that from the existence and uniqueness theorem (see Theorems \ref{exist}, \ref{unique}, \ref{strongsol}), we know that for any $\U\in\mathscr{U}_{\text{ad}}$, there exists a unique strong solution for the system \eqref{nonlin phi}-\eqref{initial conditions}. Hence, $\mathscr{A}_{\text{ad}}$ is also a nonempty set. In view of the above definition, the optimal control problem we are considering can be  formulated as:
\begin{equation}\label{control problem}\tag{OCP}
\min_{ (\u,\varphi,\U) \in \mathscr{A}_{\text{ad}}}  \mathcal{J}(\u,\varphi,\U).
\end{equation}
\begin{definition}[Optimal Solution]
	A solution to the Problem \eqref{control problem} is called an \emph{optimal solution} and the optimal triplet is denoted by $(\u^* ,\varphi^*, \U^*)$. The control $\U^*$ is called an \emph{optimal control}.
\end{definition}
In the rest of this section, we find an optimal solution to the problem (\ref{control problem}), and the optimal control is characterized  via adjoint variable. 

\subsection{\bf The Adjoint System} As is well known from the control theory literature, in order to get the necessary conditions for the existence of an optimal control to the Problem \eqref{control problem}, we need the adjoint equations corresponding to the system \eqref{nonlin phi}-\eqref{initial conditions}.  In this subsection, we formally derive the adjoint system corresponding to the problem \eqref{nonlin phi}-\eqref{initial conditions}. Let us take $\h=\mathbf{0}$ in \eqref{nonlin phi}-\eqref{nonlin u}  and define
\begin{equation}\label{n1n2}
\left\{
\begin{aligned} 
\mathscr{N}_1(\u,\varphi, \U) &:= \nu \Delta \u - (\u \cdot \nabla)\u - \nabla \widetilde{\uppi } -(\J\ast \varphi) \nabla \varphi -\nabla a\frac{\varphi^2}{2} +\U,  \\
\mathscr{N}_2(\u,\varphi) &:= -\u\cdot \nabla \varphi + \Delta (a\varphi -\J\ast \varphi + \F'(\varphi)),
\end{aligned}
\right.
\end{equation}
where $\widetilde{\uppi} = \uppi -\left( \F(\varphi) + a\frac{\varphi^2}{2}\right)$.
Then the system \eqref{nonlin phi}-\eqref{nonlin u} can be written as
$$(\partial_t \u,\partial_t \varphi) = (\mathscr{N}_1(\u,\varphi, \U), \mathscr{N}_2(\u,\varphi)).$$
We define the \emph{augmented cost functional} $\widetilde{\mathcal{J}}$ by
\begin{align}\widetilde{\mathcal{J}}(\u,\varphi,\U,\p,\eta) &:=  \int_0^T \langle \p,\partial_t \u - \mathscr{N}_1(\u,\varphi, \U) \rangle \d t+ \int_0^T \langle \eta, \partial_t \varphi - \mathscr{N}_2(\u,\varphi) \rangle \d t \no\\&\qquad- \mathcal{J}(\u,\varphi,\U) +\int_0^T \nabla q\cdot\u \ \d t , \end{align}
where $\p$ and $\eta$ denote the adjoint variables corresponding to $\u$ and $\varphi$ respectively. Corresponding to $\uppi$ in the system \eqref{nonlin phi}, we have $q$ in the adjoint system. 

Before establishing the Pontryagin maximum principle, we derive the adjoint equations formally by differentiating the augmented cost functional $\widetilde{\mathcal{J}}$ in the G\^ateaux  sense with respect to each variable. The adjoint variables $ \p, \eta $ and $\U$ satisfy the following system 
\begin{equation}\label{e2}
\left\{
\begin{aligned}
-	\p_t  - [\partial_\u \mathscr{N}_1]^*\p - [\partial_\u \mathscr{N}_2]^* \eta + \nabla q &= \mathcal{J}_\u,  \\ 
-	\eta_t - [\partial_{\varphi} \mathscr{N}_1]^*\p - [\partial_{\varphi} \mathscr{N}_2]^* \eta &= \mathcal{J}_{\varphi},  \\
- [\partial_\U \mathscr{N}_1]^*\p - [\partial_\U  \mathscr{N}_2]^* \eta &= \mathcal{J}_\U, \\ \text{div }\p&=0,\\ \p\big|_{\partial\Omega}=\frac{\partial\eta}{\partial\mathbf{n}}\Big|_{\partial \Omega}&=\mathbf{0},\\ \p(T,\cdot)=\eta(T,\cdot)&=0.
\end{aligned}
\right.
\end{equation}

Note that differentiating $\widetilde{\mathcal{J}}$ with respect to the adjoint variables recovers the original nonlinear system.
We compute $[\partial_\u \mathscr{N}_1]^*\p$, $[\partial_\u \mathscr{N}_2]^* \eta$, $[\partial_{\varphi} \mathscr{N}_1]^*\p$, $[\partial_{\varphi}  \mathscr{N}_2]^* \eta$ as
\begin{equation}
\left\{
\begin{aligned}
[\partial_\u \mathscr{N}_1]^* \p &= \nu \Delta \p +(\p \cdot \nabla^T)\u - (\u \cdot \nabla)\p,\\ [\partial_\u \mathscr{N}_2]^* \eta &= \eta\nabla \varphi,\\
[\partial_\varphi \mathscr{N}_1]^* \p &=  - \J \ast (\p \cdot \nabla \varphi) + (\nabla{\J} \ast \varphi)\cdot \p  - \nabla a\cdot \p \varphi ,\\ [\partial_\varphi \mathscr{N}_2]^* \eta &= \u \nabla \eta + a \Delta \eta - \J\ast \Delta \eta + \F''(\varphi)\Delta \eta.
\end{aligned}
\right.
\end{equation}

Also it should be noted  that the third condition in (\ref{e2}) gives $\U=-\p$ if we take $\U\in\mathscr{U}_{\text{ad}}=\mathrm{L}^2(0,T;\G_{\text{div}})$. Thus from \eqref{e2}, it follows that the adjoint variables $(\p,\eta)$ satisfy the following adjoint system: 
\begin{equation}\label{adj}
\left\{
\begin{aligned}
-\p_t- \nu \Delta \p -(\p \cdot \nabla^T)\u + (\u \cdot \nabla)\p - \eta \nabla \varphi+\nabla q &= (\u-\u_d),\\
-\eta_t   +\J \ast (\p \cdot \nabla \varphi) - (\nabla{\J} \ast  \varphi)\cdot \p+ \nabla a\cdot \p \varphi 
- \u \cdot\nabla \eta - a \Delta \eta&\\ + \J\ast \Delta \eta - \F''(\varphi)\Delta \eta &= (\varphi - \varphi_d), \\ \text{div }\p&=0,\\ \p\big|_{\partial\Omega}={\frac{\partial{\eta}}{\partial \mathbf{n}}}\Big|_{\partial \Omega}&=\mathbf{0},\\ \p(T,\cdot)=\eta(T,\cdot)&=0.
\end{aligned}
\right.
\end{equation}
The following theorem gives the unique solvability of the system (\ref{adj}) with $\p(T)=\p_T\in\G_{\text{div}}$ and $\eta(T)=\eta_T\in\mathrm{V}$. 
\begin{theorem}[Existence and Uniqueness of Adjoint System]\label{adjoint}
	Let the Assumption \ref{prop of F and J}, \eqref{fes}, \eqref{initial} along  with $\J \in \W^{2,1}(\R^2,\R)$ be satisfied. Also let us assume that $\p_T \in \G_{\text{div}}$ and $\eta_T \in \mathrm{V}$ and $(\u,\varphi)$ be a unique strong solution of the nonlinear system \eqref{nonlin phi}-\eqref{initial conditions}. Then, there exists \emph{a unique weak solution} of the system \eqref{adj} satisfying
	\begin{align}\label{space}
	(\p,\eta)\in(\mathrm{L}^{\infty}(0,T;\G_{\text{div}})\cap\mathrm{L}^2(0,T;\V_{\text{div}}))\times (\mathrm{L}^{\infty}(0,T;\mathrm{V})\cap \mathrm{L}^2(0,T;\mathrm{H}^2)),\end{align} 
	and for all $\v \in \V$ and $\zeta \in \mathrm{H}$ and for almost all $t \in (0,T)$, we have
	\begin{equation}\label{0.1}
	\left\{
	\begin{aligned}
	- \langle \p_t, \v \rangle + \nu(\nabla \p,\nabla \v) -((\p \cdot \nabla^T)\u,\v) + ((\u \cdot \nabla)\p,\v) - (\eta \nabla \varphi,\v) &= (\u-\u_d,\v), \\ 
	-(\eta_t,\zeta) + (\J \ast (\p \cdot \nabla \varphi),\zeta) - ((\nabla{\J} \ast  \varphi)\cdot \p,\zeta) + (\nabla a\cdot \p \varphi 
	,\zeta)\qquad \ \, \, \\  - (\u \cdot\nabla \eta,\zeta) - (a \Delta \eta,\zeta)  +(\J\ast \Delta \eta,\zeta) - (\F''(\varphi)\Delta \eta,\zeta) &= (\varphi - \varphi_d,\zeta),  
	\end{aligned}	
	\right.
	\end{equation}
	where $\p(T)=\p_T\in\G_{\text{div}}$,  $\eta(T)=\eta_T\in\mathrm{V}$ are satisfied in the weak sense. 
\end{theorem}
\begin{proof}
Using Galerkin approximation technique, we can proof as in Theorem 3.5 of x\cite{BDM2}.
\end{proof}

\subsection{Existence of an Optimal Control}

Let us now show that an optimal triplet $(\u^*,\varphi^*,\U^*)$ exists for the problem \eqref{control problem}.
\begin{theorem}[Existence of an Optimal Triplet]\label{optimal}
	Let the Assumption \ref{prop of F and J} along with the condition (\ref{fes}) holds true and the initial data $(\u_0,\varphi_0)$ satisfying (\ref{initial}) be given. If $\mathrm{J}\in\mathrm{W}^{2,1}(\mathbb{R}^2;\mathbb{R})$, then there exists at least one triplet  $(\u^*,\varphi^*,\U^*)\in\mathscr{A}_{\text{ad}}$  such that the functional $ \mathcal{J}(\u,\varphi,\U)$ attains its minimum at $(\u^*,\varphi^*,\U^*)$, where $(\u^*,\varphi^*)$ is the unique strong solution of \eqref{nonlin phi}-\eqref{initial conditions}  with the control $\U^*$.
\end{theorem}

\begin{proof}
	\noindent\textbf{Claim (1):} \emph{There exists  an optimal triplet $(\u^*,\varphi^*,\U^*)\in\mathscr{A}_{\text{ad}}$.}	
	Let us define $$\mathscr{J} := \inf \limits _{\U \in \mathscr{U}_{\text{ad}}}\mathcal{J}(\u,\varphi,\U).$$
	Since, $0\leq \mathscr{J} < +\infty$, there exists a minimizing sequence $\{\U_n\} \in \mathscr{U}_{\text{ad}}$ such that $$\lim_{n\to\infty}\mathcal{J}(\u_n,\varphi_n,\U_n) = \mathscr{J},$$ where $(\u_n,\varphi_n)$ is the unique strong solution of \eqref{nonlin phi}-\eqref{initial conditions} with the control $\U_n$ and 
	\begin{align}\label{initial1}
	\u_n(0)=	\u_0 \in\V_{\text{div}}\ \text{ and }\ \varphi_n(0)=\varphi_0 \in \mathrm{H}^2(\Omega).
	\end{align}
	Without loss of generality, we assume that $\mathcal{J}(\u_n,\varphi_n,\U_n) \leq \mathcal{J}(\u,\varphi,\mathbf{0})$, where $(\u,\varphi,\mathbf{0})\in\mathscr{A}_{\text{ad}}$. From the definition of $\mathcal{J}(\cdot,\cdot,\cdot)$, this implies
	\begin{align}\label{bound}
	&\frac{1}{2} \int_0^T \|\u_n(t)-\u_d(t)\|^2 \d t+ \frac{1}{2} \int_0^T \|\varphi_n(t) -\varphi_d(t)\|^2 \d t+ \frac{1}{2} \int_0^T\|\U_n(t)\|^2\d t\nonumber\\& \leq \frac{1}{2} \int_0^T \|\u(t)-\u_d(t)\|^2\d t + \frac{1}{2} \int_0^T \|\varphi(t) -\varphi_d(t)\|^2\d t . \end{align}
	Since $\u,\u_d \in \mathrm{L}^{2}(0,T;\G_{\text{div}})$ and $\varphi,\varphi_d \in \mathrm{L}^{2}(0,T;\mathrm{H})$, from the above relation, it is clear that, there exist a $K>0$, large enough such that
	$$0 \leq \mathcal{J}(\u_n,\varphi_n,\U_n) \leq K < +\infty.$$
	In particular, there exists a large $C>0,$ such that
	$$ \int_0^T \|\U_n(t)\|^2 \d t \leq  C < +\infty .$$
	Therefore the sequence $\{\U_n\}$ is uniformly bounded in the space $\mathrm{L}^2(0,T;\G_{\text{div}})$. Since $(\u_n,\varphi_n)$ is a unique weak solution of the system \eqref{nonlin phi}-\eqref{initial conditions} with control $\U_n$, from the energy estimates, one can easily show that the sequence $\{\u_n\} $ is uniformly bounded in $\mathrm{L}^{\infty}(0,T;\G_{\text{div}})\cap \mathrm{L}^2(0,T;\V_{\text{div}})$ and $\{\varphi_n\} $ is uniformly bounded in $\mathrm{L}^{\infty}(0,T;\mathrm{H})\cap\mathrm{L}^2(0,T;\mathrm{V})$. 
	Hence, by using the Banach-Alaglou theorem, we can extract a subsequence $\{(\u_{n},\varphi_{n}, \U_n)\}$ such that 
	\begin{equation}\label{conv}
	\left\{
	\begin{aligned} 
	\u_n&\xrightharpoonup{w^*}\u^*\text{ in } \mathrm{L}^{\infty}(0,T;\G_{\text{div}}), \\ \u_n &\rightharpoonup \u^*\text{ in  }\textrm{L}^2(0,T;\V_{\text{div}}),\\ 
	(\u_n)_t &\rightharpoonup \u^*_t\text{ in  }\textrm{L}^2(0,T;\V_{\text{div}}'),\\\varphi_n &\xrightharpoonup{w^*} \varphi^*\text{ in }\mathrm{ L}^{\infty}(0,T;\mathrm{H}),\\ 
	(\varphi_{n}) _t&\rightharpoonup \varphi^*_t\text{ in }\mathrm{ L}^2(0,T;\mathrm{V}'),\\ \U_n &\rightharpoonup  \U^*\text{ in  }\mathrm{L}^2(0,T;\G_{\text{div}}).
	\end{aligned}
	\right.
	\end{equation}
	A calculation similar to the proof of Theorem \ref{exist} , (see Theorem 2, \cite{weak}) and Theorem \ref{unique} (see Theorem 2, \cite{unique}) and using Aubin-Lion's compactness theorem  and the convergence in \eqref{conv}, we get  
	\begin{equation}
	\left\{
	\begin{aligned}
	\u_n&\to \u^*\text{ in }\mathrm{L}^{2}(0,T;\G_{\text{div}}), \text{ a. e. }\text{ in }\Omega\times(0,T),\\
	\varphi_n &\to \varphi^*\text{ in }\mathrm{ L}^{2}(0,T;\mathrm{H}), \text{ a. e. }\text{ in }\Omega\times(0,T).
	\end{aligned}
	\right. 
	\end{equation}
	Proceeding similarly as in Theorem 1, \cite{weak} and Theorem 2, \cite{unique}, we obtain  $(\u^*,\varphi^*)$ is a unique weak solution of \eqref{nonlin phi}-\eqref{initial conditions} with control $\U^*$. Also, $(\u^*,\varphi^*) \in \mathrm{C}([0,T];\G_{\text{div}}) \times \mathrm{C}([0,T];\mathrm{H})$. Note that the initial condition (\ref{initial1}) and (\ref{fes}) gives the following convergences:
	\begin{equation*}
	\u_n\in \mathrm{L}^{\infty}(0,T;\V_{\text{div}})\cap \mathrm{L}^2(0,T;\H^2(\Omega)), \ \ \varphi_n\in \mathrm{L}^{\infty}((0,T) \times\Omega)\times \mathrm{L}^{\infty}(0,T;\mathrm{V})).
	\end{equation*}
	and 
	\begin{equation*}\varphi_n \in\mathrm{L}^{\infty}(0,T;\mathrm{W}^{1,p}),\ 2\leq p<\infty.
	\end{equation*}
	Thus, we have (see Theorem 2, \cite{strong} also)
	\begin{equation}
	\left\{
	\begin{aligned}
	\u_n&\xrightharpoonup{w^*}\u^*\ \text{ in }\ \mathrm{L}^{\infty}(0,T;\V_{\text{div}}), \\
	\u_n&\rightharpoonup\u^*\ \text{ in }\ \mathrm{L}^2(0,T;\H^2(\Omega)),\\
	(\u_n)_t&\rightharpoonup\u^*_t\ \text{ in }\ \mathrm{L}^{2}(0,T;\G_{\text{div}}), \\
	\varphi_n&\to\varphi^*\ \text{ a.e., }(x,t) \ \in\Omega\times(0,T), \\
	(\varphi_{n})_t&\rightharpoonup\varphi^*_t\ \text{ in }\ \mathrm{L}^{2}(0,T;\mathrm{H})
	\end{aligned}
	\right.
	\end{equation}
	Thus the above convergences and Remark \ref{rem2.12} (see \eqref{ues}) imply that $(\u^*,\varphi^*)$ has the regularity given in (\ref{0.22}) and (\ref{0.23}). Since $\u^*\in\mathrm{C}([0,T];\V_{\text{div}})$ and $\mathrm{J}\in\mathrm{W}^{2,1}(\mathbb{R}^2;\mathbb{R}),$ we know that $\varphi^*\in\mathrm{C}([0,T];\mathrm{H}^2)$ and hence we have 
	\begin{align}\label{initial2}
	\u^*(0)=	\u_0\in\V_{\text{div}}\ \text{ and }\ \varphi^*(0)=\varphi_0\in \mathrm{H}^2(\Omega).
	\end{align}
	Hence $(\u^*,\varphi^*)$ is a unique strong solution of  \eqref{nonlin phi}-\eqref{initial conditions} with control $\U^*\in\mathrm{L}^2(0,T;\G_{\text{div}})$. Remember that $\mathscr{U}_{\text{ad}}$ is the space consisting of controls $\U \in \mathrm{L}^2(0,T;\G_{\text{div}})$ and hence $\mathscr{U}_{\text{ad}}\ni\U_n \rightharpoonup  \U^*\text{ in  }\mathrm{L}^2(0,T;\G_{\text{div}})$ implies that $\U^*\in\mathscr{U}_{\text{ad}}$. This easily gives  $(\u^*,\varphi^*,\U^*)\in \mathscr{A}_{\text{ad}}$. 
	
	\noindent \textbf{Claim (2):} \emph{$\mathscr{J}=\mathcal{J}(\u^*,\varphi^*,\U^*)$}. Since the cost functional $\mathcal{J}(\cdot,\cdot,\cdot)$ is continuous and convex on $\mathrm{L}^2(0,T;\G_{\text{div}}) \times \mathrm{L}^2(0,T;\mathrm{H}) \times \mathscr{U}_{\text{ad}}$, it follows that $\mathcal{J}(\cdot,\cdot,\cdot)$ is weakly lower semi-continuous (see Proposition 1, Chapter 5, \cite{AbEk}). That is, for a sequence $\U_n \rightarrow\U\in\mathscr{U}_{\text{ad}}$, 
	$$(\u_n,\varphi_n,\U_n)\xrightharpoonup{w}(\u^*,\varphi^*,\U^*)\text{ in }\mathrm{L}^2(0,T;\V_{\text{div}})\times \mathrm{L}^2(0,T;\mathrm{V})\times \mathrm{L}^2(0,T;\G_{\text{div}}),$$
	we have 
	\begin{align*}
	\mathcal{J}(\u^*,\varphi^*,\U^*) \leq  \liminf \limits _{n\rightarrow \infty} \mathcal{J}(\u_n,\varphi_n,\U_n).
	\end{align*}
	Therefore, we get
	\begin{align*}\mathscr{J} \leq \mathcal{J}(\u^*,\varphi^*,\U^*) \leq  \liminf \limits _{n\rightarrow \infty} \mathcal{J}(\u_n,\varphi_n,\U_n)=  \lim \limits _{n\rightarrow \infty} \mathcal{J}(\u_n,\varphi_n,\U_n) = \mathscr{J},\end{align*}
	and hence $(\u^*,\varphi^*,\U^*)$ is a minimizer.
\end{proof}
\subsection{\bf Pontryagin Maximum Principle}
In this subsection, we prove the Pontryagin maximum principle  for the  optimal control problem defined in \eqref{control problem}. Pontryagin Maximum principle  gives a first order necessary condition for the optimal control problem  \eqref{control problem}. We also characterize the optimal control in terms of the adjoint variables. Even though we announced the subsection title as Pontryagin maximum principle, our problem is a minimization of the cost functional given in (\ref{cost}) and hence we obtain  a  minimum principle. 

The following minimum principle is satisfied by the optimal triplet \\$(\u^*,\varphi^*,\U^*)\in\mathscr{A}_{\text{ad}}$:
\begin{equation}\label{pm}
\frac{1}{2}\|\U^*(t)\|^2+(\p(t),\U^*(t) ) \leq \frac{1}{2}\|\mathrm{W}\|^2+( \p(t),\mathrm{W} ),
\end{equation} 
for all  $\mathrm{W}\in \G_{\text{div}},$ and a.e. $t\in[0,T]$. 
Equivalently the above minimum principle may be  written in terms of the \emph{Hamiltonian formulation}. 
Let us first define the \emph{Lagrangian} by
$$\mathscr{L}(\u,\varphi,\U) = \frac{1}{2} (\|\u-\u_d\|^2 + \|\varphi - \varphi_d\|^2+ \|\U\|^2).$$
Then, we can define the corresponding \emph{Hamiltonian} by
$$\mathscr{H}(\u,\varphi,\U,\p,\eta)= \mathscr{L}(\u,\varphi,\U) + (\p, \mathscr{N}_1(\u,\varphi,\U) ) + (\eta, \mathscr{N}_2(\u,\varphi)),$$ where $\mathscr{N}_1$ and $\mathscr{N}_2$ are defined by (\ref{n1n2}).
Hence, we get the minimum principle as
\begin{equation}
\mathscr{H}(\u^*(t),\varphi^*(t),\U^*(t),\p(t),\eta(t)) \leq \mathscr{H}(\u^*(t),\varphi^*(t),\mathrm{W},\p(t),\eta(t)),
\end{equation}
for all $\mathrm{W} \in \G_{\text{div}}$ and a.e. $t\in[0,T]$. 
\begin{definition}[Subgradient, Subdifferential] Let $\X$ be a real Banach space and $f:\X\to(-\infty,\infty]$  a functional on $\X$. A linear functional $u'\in\X'$  is called \emph{subgradient} of $f$ at $u$ if $f(u)\neq +\infty$ and
	$$f(v)\geq f(u)+\langle u',v-u\rangle_{\X'\times \X},$$ holds for all $v \in\X$. The set of all sub-gradients of $f$ at u is called \emph{subdifferential} $\partial f(u)$ of $f$ at $u$.
	\end{definition}
We say that $f$ is \emph{G\^ateaux differentiable} at $u$ in $\X$ if $\partial f(u)$ consists of exactly one element, which we denote by $f_u(u)$.   This is equivalent to the assertion that the limit $$\langle f_u(u),h\rangle_{\X'\times\X}=\lim_{\tau\to 0}\frac{f(u+\tau h)-f(u)}{\tau}=\frac{\d}{\d\tau}f(u+\tau h)\bigg|_{\tau=0},$$ exists for all $h\in\X$. 

From \eqref{pm}, we see that $-\p \in \partial \frac{1}{2}\|{\U^*}(t)\|^2$, where $\partial$ denotes the subdifferential. Since, $\frac{1}{2}\|\cdot\|^2$ is G\^ateaux differentiable, the subdifferential consists of a single point and it follows that
\begin{equation}\label{nec}
-\p(t) = \U^*(t), \ \text{ a.e. }\  t \in [0,T].
\end{equation}

Now we state the main result of our paper. For similar results regrading the incompressible Navier-Stokes equations, see for example \cite{sritharan,FART} and for the linearized compressible Navier-Stokes equations, see \cite{SDMTS}. 
\begin{theorem}[Pontryagin Minimum Principle]\label{main}
	Let $(\u^*,\varphi^*,\U^*)\in\mathscr{A}_{\text{ad}}$ be the optimal solution of the Problem \ref{control problem} obtained in Theorem \ref{optimal}. Then there exists \emph{a unique weak solution} $(\p,\eta)$ of the adjoint system \eqref{adj} such that

\begin{equation}\label{3.41}
\frac{1}{2}\|\U^*(t)\|^2 +(\p(t),\U^*(t) ) \leq \frac{1}{2}\|\mathrm{W}\|^2 +(\p(t),\mathrm{W} ),  
\end{equation}
for all $\mathrm{W} \in \G_{\text{div}}$ and  almost every $t \in [0,T]$.
\end{theorem}
{\it Proof}
	Let $(\u^*,\varphi^*,\U^*)\in\mathscr{A}_{\text{ad}}$ be the optimal triplet of the control problem \eqref{control problem}. Let $\mathcal{F}(\U)=\mathcal{J}(\u_{\U},\varphi_{\U},{\U})$, where $({\u}_{\U},\varphi_{\U},\U)$ is the solution of the system \eqref{nonlin phi}-\eqref{initial conditions} with control $\U\in\G_{\text{div}}$. Let $\U^*+\lambda\U\in\G_{\text{div}}$ such that $(\u_{\u^*+\lambda\U},\varphi_{\u^*+\lambda\U},\U_{\u^*+\lambda\U})\in\mathscr{A}_{\text{ad}}$, for all $0\leq \lambda \leq 1$. Then, for $\lambda\in[0,1]$, we can deduce
	\begin{align}\label{lam}
	&	\mathcal{F}(\U^*+ \lambda \U) - \mathcal{F}(\U^*)\nonumber \\&= \frac{1}{2}\int_0^T  \|\u_{\U^*+\lambda \U} (t)- \u_{\U^*}(t)\|^2 \d t+ \int_0^T (\u_{\U^*+\lambda \U} (t)- \u_{\U^*}(t),\u_{\U^*}(t) -\u_d(t))\d t \nonumber \\
	&\quad+\frac{1}{2}\int_0^T  \|\varphi_{\U^*+\lambda \U}(t) - \varphi_{\U^*}(t)\|^2 \d t+ \int_0^T (\varphi_{\U^*+\lambda \U}(t) - \varphi_{\U^*}(t),\varphi_{\U^*} (t)-\varphi_d(t))\d t\nonumber\\
	& \quad+ \frac{1}{2} \int_0^T  \lambda^2 (\U(t),\U(t))\d t +  \int_0^T  \lambda (\U(t), \U^*(t))\d t. 
	\end{align}
	Since $(\u_{\U^*+\lambda \U},\varphi_{\U^*+\lambda \U})$ and $(\u_{\U^*},\varphi_{\U^*})$ are the unique strong solutions of the system \eqref{nonlin phi}-\eqref{initial conditions} with controls $\U^*+\lambda \U$ and $\U^*$ respectively, using the estimates given in the Uniqueness Theorem (see for example Theorem \ref{unique}), $\|\u_{\U^*+\lambda \U} - \u_{\U^*}\|^2_{\mathrm{L}^2(0,T;\V_{\text{div}})}$ can be estimated by $\lambda^2 \|\U\|^2_{\mathrm{L}^{2}(0,T;\G_{\text{div}})}$. Thus dividing by $\lambda$, and then sending $\lambda \rightarrow 0$, we have  $\|\u_{\U^*+\lambda \U} - \u_{\U^*}\|_{\mathrm{L}^2(0,T;\G_{\text{div}})} \rightarrow 0$. Similarly  $\|\varphi_{\U^*+\lambda \U} - \varphi_{\U^*}\|_{\mathrm{L}^2(0,T;\mathrm{V}')} \rightarrow 0$ as $\lambda \rightarrow 0$. 
	
	Let us denote the G\^ateaux derivative of $\mathcal{F}$ at $\U^*$ in the direction of $\U\in\G_{\text{div}}$ by $(\mathcal{F}'(\U^*),\U)$. Let $(\w,\psi)$ satisfy the linearized system \eqref{lin w}-\eqref{lin initial conditions} with control $\U$, and initial data and forcing term to be equal to zero, that is, $\w(0)=\widetilde{\h}=\mathbf{0}$ and $\psi(0)=0$.  From Lemma \ref{lem3.8} (see below), we have 
	\begin{align}\label{conv1}
	\lim_{\lambda\to0}\frac{\|\u_{\U^*+\lambda \U} - \u_{\U^*}-\lambda\w\|_{\mathrm{L}^2(0,T;\V_{\text{div}})}}{|\lambda|}=0,
	\end{align}
	and 
	\begin{align}\label{conv2}
	\lim_{\lambda\to0}\frac{\|\varphi_{\U^*+\lambda \U} - \varphi_{\U^*}-\lambda\psi\|_{\mathrm{L}^2(0,T;\mathrm{H})}}{|\lambda|}=0,
	\end{align}
	since $\u_{\U^*}-\u_d\in\mathrm{L}^2(0,T;\V_{\text{div}})\subset\mathrm{L}^2(0,T;\G_{\text{div}})\subset \mathrm{L}^2(0,T;\V_{\text{div}}')$ and $\varphi_{\U^*}-\varphi_d\in\mathrm{L}^2(0,T;\mathrm{V})$. Dividing by $\lambda$ and then taking $\lambda \rightarrow 0$ in (\ref{lam}), we obtain
	\begin{align*}
	0 &\leq (\mathcal{F}'(\U^*), \U) = \underset{\lambda \rightarrow 0}{\lim} \frac{\mathcal{F}(\U^*+ \lambda \U)- \mathcal{F}(\U^*)}{|\lambda|} \\
	&= \int_0^T  (\w(t),\u_{\U^*}(t) -\u_d(t))\d t + \int_0^T  (\psi(t),\varphi_{\U^*}(t) -\varphi_d(t)) \d t+ \int_0^T (\U(t),\U^*(t)) \d t,
	\end{align*} 
	where
	\begin{align}
	\label{4.16}
	\w=\lim_{\lambda\to0}\frac{\u_{\U^*+\lambda \U} - \u_{\U^*}}{|\lambda|}\text{ and }\psi=\lim_{\lambda\to0}\frac{\varphi_{\U^*+\lambda \U} - \varphi_{\U^*}}{|\lambda|}.
	\end{align}
	Identifying the inner product in $\G_{\text{div}}$  with the duality pairing between $\V_{\text{div}}$ and $\V_{\text{div}}'$, using (\ref{adj}), we obtain 
	\begin{align*}
	0 \leq (\mathcal{F}'(\U^*) , \U )&= \int_0^T  (\w,-\p_t-\nu \Delta \p -(\p \cdot \nabla^T)\u^*+ (\u^* \cdot \nabla)\p -  \eta\nabla \varphi^*-\nabla q)\d t \\
	&\quad +\int_0^T  (\psi, -\eta_t + \J \ast (\p \cdot \nabla \varphi^*) - (\nabla\J \ast \varphi^*)\cdot\p +  \nabla a\cdot \p \varphi^*)\d t \\
	&\quad + \int_0^T (\psi, -\u^* \cdot\nabla \eta - a \Delta \eta + \J\ast \Delta \eta - \F''(\varphi^*)\Delta \eta)\d t+ \int_0^T (\U,\U^*)\d t.
	\end{align*}
	Since  $\nabla\cdot\w=0$ and $\nabla\cdot\p=0$, an integration by parts yields 
	\begin{align*}
	0 \leq &(\mathcal{F}'(\U^*), \U )= \int_0^T  (\w_t - \nu \Delta \w +(\w \cdot \nabla)\u^* + (\u^* \cdot \nabla)\w + \nabla a\varphi^* \psi,\p)\d t  \\
	&\quad+\int_0^T   ((\J \ast \psi) \nabla \varphi^*+(\J \ast \varphi^*) \nabla \psi-\nabla\widetilde{\uppi}_{\w},\p)\d t  \\
	&\quad+ \int_0^T (\psi_t+\nabla \varphi^*\cdot \w+\u^* \cdot \nabla \psi-  \Delta (a \psi- \J\ast  \psi- \F''(\varphi^*) \psi), \eta) \d t+ \int_0^T (\U,\U^*) \d t \\ 
	&= \int_0^T (\U(t),\p(t)) \d t+ \int_0^T (\U(t),\U^*(t))\d t,
	\end{align*}  
	where  last equality follows; thanks to the equation satisfied by $\w$ with control $\U$ and  $\widetilde{\h}=\mathbf{0}$. Thus, we have 
	$$0 \leq (\mathcal{F}'(\U^*(t)), \U(t)) = \int_0^T  (\U(t),\p(t))\d t + \int_0^T  (\U(t),\U^*(t))\d t.$$ 
	Similarly if we take the directional derivative of $\mathcal{F}$ in the direction of $-\U\in\G_{\text{div}}$, we  obtain                                                                                                                                                                                                                                                                            
	$(\mathcal{F}'(\U^*(t)), \U(t)) \leq 0.$
	Hence, we obtain 
	$(\mathcal{F}'(\U^*(t)) , \U(t) )= 0,$ and we have 
	\begin{align}\label{4.51}
	\int_0^T  (\U(t),\p(t)) \d t+ \int_0^T  (\U(t),\U^*(t)) \d t&= 0, \end{align} for all $\U\in\G_{\text{div}}$. Thus it is immediate that 
	\begin{align*}
	(\U(t),\p(t) + \U^*(t)) &= 0, \ \text{ for all } \ \U \in \G_{\text{div}}\  \text{  and  a.e. }\  t \in [0,T].
	\end{align*}
	Since the above equality is true for all $\U \in \G_{\text{div}}$, we get  
\begin{equation*}
 \U^*(t)=-\p(t), \ \text{ a.e. }\  t \in [0,T].
\end{equation*}

\begin{lemma}\label{lem3.8}
	Let $(\u_0, \varphi_0)$ satisfies (\ref{initial}) and $\F(\cdot)$ satisfies \ref{fes}, the mapping $\U \mapsto (\u_{\U},\varphi_{\U})$  from $\mathscr{U}_{\text{ad}}$ into  $\left(\mathrm{C}([0,T];\G_{\text{div}})\cap\mathrm{L}^{2}(0,T;\V_{\text{div}})\right) \times \left(\mathrm{C}([0,T];\mathrm{V}')\cap\mathrm{L}^{2}(0,T;\mathrm{H})\right)$ is G\^ateaux differentiable. Furthermore, we have 
	\begin{align}
	\label{limit}
	\left( \lim_{\lambda\to0}\frac{\u_{\U^*+\lambda \U} - \u_{\U^*}}{|\lambda|}, \lim_{\lambda\to0}\frac{\varphi_{\U^*+\lambda \U} - \varphi_{\U^*}}{|\lambda|} \right) = (\w , \psi), \end{align}
	where $(\w, \psi)$ is the \emph{unique weak solution} of 
	\begin{eqnarray}\label{4.17a}
	\left\{
	\begin{aligned} 
	\w_t - \nu \Delta \w + (\w \cdot \nabla )\u_{\U^*} &+ (\u_{\U^*} \cdot \nabla )\w + \nabla \widetilde{\uppi}_\w \\ &= -\nabla a\psi \varphi_{\U^*} -(\J\ast \psi) \nabla \varphi_{\U^*}  - (\J \ast \varphi_{\U^*}) \nabla \psi  +\U,\ \text{ in } \ \Omega\times(0,T),  \\
	\psi_t + \w\cdot \nabla \varphi_{\U^*} + \u_{\U^*} \cdot\nabla \psi &= \Delta \widetilde{\mu}, \ \text{ in }\ \Omega\times(0,T),   \\ 
	\widetilde{\mu} &= a \psi - \J\ast \psi + \F''(  \varphi_{\U^*}) \psi,\ \text{ in }\ \Omega\times(0,T),   \\
	\text{div }\w &= 0,\ \text{ in }\ \Omega\times(0,T),   \\
	\frac{\partial \widetilde{\mu}}{\partial\mathbf{n}} &= \mathbf{0}, \ \w=\mathbf{0}  \ \text{on } \ \partial \Omega \times (0,T),  \\
	\w(0) &= \mathbf{0}, \ \psi(0) = 0 \ \text{ in } \ \Omega. 
	\end{aligned}
	\right.
	\end{eqnarray}
	That is, we have 
		\begin{align*}
	\lim_{\lambda\to0}\frac{\|\u_{\U^*+\lambda \U} - \u_{\U^*}-\lambda\w\|_{\mathrm{L}^2(0,T;\V_{\text{div}})}}{|\lambda|}=0,\
	\lim_{\lambda\to0}\frac{\|\varphi_{\U^*+\lambda \U} - \varphi_{\U^*}-\lambda\psi\|_{\mathrm{L}^2(0,T;\mathrm{H})}}{|\lambda|}=0,
	\end{align*}
and the pair $(\u_{\U^*},\varphi_{\U^*})$  and $(\u_{\U^*+\lambda\U},\varphi_{\U^*+\lambda\U})$ are the unique strong solutions of the system (\ref{nonlin phi})-(\ref{initial conditions}) with  controls $\U^*$ and $\U^*+\lambda\U$, respectively.
\end{lemma}
\begin{proof}
	In order to prove (\ref{limit}), we need to prove (\ref{conv1}) and (\ref{conv2}). 
	Let us set $$(\doublewidetilde{\u},\doublewidetilde{\varphi}) = (\u_{\U^*+\lambda \U} - \u_{\U^*}  , \varphi_{\U^*+\lambda \U} - \varphi_{\U^*} )$$
	$$ \text{and} \; \; (\y, \varrho) = (\doublewidetilde{\u} -\lambda \w , \doublewidetilde{\varphi} -\lambda \psi).$$
	Observe that, $(\doublewidetilde{\u},\doublewidetilde{\varphi})$ satisfies the following system:
	\begin{eqnarray}\label{4.18a}
	\left\{
	\begin{aligned}
	{\doublewidetilde{\u}}_t - \nu \Delta \doublewidetilde{\u} + (\doublewidetilde{\u} \cdot \nabla )\doublewidetilde{\u}&+ (\doublewidetilde{\u} \cdot \nabla )\u_{\U^*} + (\u_{\U^*} \cdot \nabla )\doublewidetilde{\u} + \nabla \doublewidetilde{\uppi}_{\doublewidetilde{\u}} \\&= - \frac{\nabla a}{2} \doublewidetilde{\varphi}^2 - \nabla a \doublewidetilde{\varphi} \varphi_{\U^*} -(\J\ast \doublewidetilde{\varphi}) \nabla \doublewidetilde{\varphi}, -(\J\ast \doublewidetilde{\varphi}) \nabla \varphi_{\U^*}\\
	& \quad -(\J \ast \varphi_{\U^*} ) \nabla \doublewidetilde{\varphi}  + \lambda \U, \ \text{ in } \ \Omega\times(0,T),\\
	{\doublewidetilde{\varphi}}_t + \doublewidetilde{\u} \cdot \nabla \doublewidetilde{\varphi}+ \doublewidetilde{\u} \cdot \nabla \varphi_{\U^*} + \u_{\U^*} \cdot\nabla \doublewidetilde{\varphi} &= \Delta \doublewidetilde{\mu}, \ \text{ in } \ \Omega\times(0,T), \\ 
	\doublewidetilde{\mu} &= a \doublewidetilde{\varphi} - \J\ast \doublewidetilde{\varphi} + \mathrm{F}'(\varphi_{\U^*+\lambda\U})-\mathrm{F}'(\varphi_{\U^*}),  \\
	\text{div }\doublewidetilde{\u} &= 0, \ \text{ in } \ \Omega\times(0,T), \\
	\frac{\partial \doublewidetilde{\mu}}{\partial\mathbf{n}} &= \mathbf{0}, \ \doublewidetilde{\u}=\mathbf{0}  \ \text{on } \ \partial \Omega \times (0,T),\\
	\doublewidetilde{\u}(0) &= \mathbf{0}, \ \doublewidetilde{\varphi}(0) = 0 \ \text{ in } \ \Omega,
	\end{aligned}
	\right.
	\end{eqnarray}
	Let us now use Taylor's series expansion of $\mathrm{F}'(\varphi_{\U^*+\lambda \U}) - \mathrm{F}'(\varphi_{\U^*})$ up to third order to obtain 
	\begin{equation*}
\mathrm{F}'(\varphi_{\U^*+\lambda \U}) - \mathrm{F}'(\varphi_{\U^*})=\F''(\varphi_{\U^*})  \doublewidetilde{\varphi} + \F'''(\varphi_{\U^*} + \theta \doublewidetilde{\varphi})\frac{\doublewidetilde{\varphi}^2}{2},
\end{equation*}
for some $0<\theta<1$. 
	Using the decomposition of $\mu\nabla\varphi$ given in \eqref{mh}, we know that $\widetilde{\uppi}_{\U^*}=\uppi_{\U^*}-[\mathrm{F}(\varphi_{\U^*})+\frac{a}{2}\varphi_{\U^*}^2]$ and hence we  have 
	\begin{align}\label{pcal}\doublewidetilde{\uppi}_{\doublewidetilde{\u}}=\uppi_{\U^*+\lambda \U}-\uppi_{\U^*}-\left[\F(\varphi_{\U^*+\lambda\U})-\F(\varphi_{\U^*})+\frac{a}{2}\left(\varphi_{\U^*+\lambda \U}^2-\varphi_{\U^*}^2\right)\right].\end{align} 

	Moreover, it can be shown that $(\y, \varrho)$ satisfies the following system:
	\begin{eqnarray}\label{4.44}
	\left\{
	\begin{aligned}
	\y_t - \nu \Delta \y + (\y \cdot \nabla )\u_{\U^*} + (\u_{\U^*} \cdot \nabla )\y &+ \nabla \uppi_\y + (\J\ast \varrho) \nabla \varphi_{\U^*} + (\J \ast \varphi_{\U^*}) \nabla \varrho  - \nabla a \varrho \varphi  \\ &= -(\doublewidetilde{\u} \cdot \nabla )\doublewidetilde{\u} 
	- \frac{\nabla a}{2} \doublewidetilde{\varphi}^2 -(\J\ast \doublewidetilde{\varphi}) \nabla \doublewidetilde{\varphi} ,\ \text{ in } \ \Omega\times(0,T),  \\ 
	\varrho_t + \y \cdot \nabla \varphi_{\U^*} + \u_{\U^*} \cdot \nabla \varrho  - \Delta \widetilde{\mu}_{\varrho} &= - \doublewidetilde{\u} \cdot \nabla\doublewidetilde{\varphi}, \ \text{ in } \ \Omega\times(0,T), \\ 
	\widetilde{\mu}_{\varrho} &= a \varrho - \J\ast \varrho + \F''(\varphi_{\U^*}) \varrho + \F'''(\varphi_{\U^*} + \theta \doublewidetilde{\varphi})\frac{\doublewidetilde{\varphi}^2}{2},\\
	\text{div }\y &= 0, \ \text{ in } \ \Omega\times(0,T), \\
	\frac{\partial \widetilde{\mu}_{\varrho}}{\partial\mathbf{n}} &= \mathbf{0}, \ \y=\mathbf{0}  \ \text{on } \ \partial \Omega \times (0,T),\\
	\y(0) &= \mathbf{0}, \ \varrho(0) =0  \ \text{ in } \ \Omega,
	\end{aligned}
	\right.
	\end{eqnarray}
	where \begin{align*}\uppi_{\y}&=(\uppi_{\U^*+\lambda \U}-\uppi_{\U^*})-\lambda\widetilde{\uppi}\no\\&\quad-\left[\F(\varphi_{\U^*+\lambda\U})-\F(\varphi_{\U^*})-\lambda\mathrm{F}'(\varphi_{\U^*})\psi+\frac{a}{2}\left(\varphi_{\U^*+\lambda \U}^2-\varphi_{\U^*}^2\right)-\lambda a\varphi_{\U^*}\psi\right]\\&=\Upsilon-\varrho(\mathrm{F}'(\varphi_{\U^*})+a\varphi_{\U^*})-\frac{\doublewidetilde{\varphi}^2}{2}\left(a+\F''(\varphi_{\U^*}+\theta_1\doublewidetilde{\varphi})\right),\end{align*}  for some $0<\theta_1<1$. In the above estimate, we used the Taylor series expansion and we defined $\Upsilon:=(\uppi_{\U^*+\lambda \U}-\uppi_{\U^*})-\lambda\widetilde{\uppi}$. 
	
	Let us denote the nonlinear terms in the equations for $\y$, $\varrho$ in\eqref{4.44} by $\mathbf{k}(x,t)$, $l(x,t)$ respectively. That is, we have 
	\begin{equation}\label{4.58}
	\left\{
	\begin{aligned}
	\mathbf{k}(x,t) &= - (\doublewidetilde{\u} \cdot \nabla )\doublewidetilde{\u} - \frac{\nabla a}{2} \doublewidetilde{\varphi}^2  -(\J\ast \doublewidetilde{\varphi}) \nabla \doublewidetilde{\varphi} - \nabla \uppi_\y, \\
	l(x,t) &= - \doublewidetilde{\u} \cdot \nabla\doublewidetilde{\varphi}+\frac{1}{2} \Delta (\F'''(\varphi_{\U^*} + \theta \doublewidetilde{\varphi})\doublewidetilde{\varphi}^2).
	\end{aligned}
	\right.
	\end{equation}
	Our next aim is to show that
	\begin{align}
	&\|\mathbf{k}(x,t) \|_{\mathrm{L}^{2}(0,T;\V_{\text{div}}')}  \leq C\left( \|\doublewidetilde{\u}\|_{\mathrm{L}^{\infty}(0,T;\G_{\text{div}})}\|\doublewidetilde{\u}\|_{\mathrm{L}^2(0,T;\V_{\text{div}})}+\|\doublewidetilde{\varphi} \|_{\mathrm{L}^{\infty}(0,T;\mathrm{H})}\|\doublewidetilde{\varphi}\|_{\mathrm{L}^2(0,T;\mathrm{V})}\right), \label{k} \\
	&\|l(x,t)\|_{\mathrm{L}^{2}(0,T;\mathrm{H}^{-2})}  \leq C\left( \|\doublewidetilde{\u}\|_{\mathrm{L}^{\infty}(0,T;\G_{\text{div}})}\|\doublewidetilde{\u}\|_{\mathrm{L}^2(0,T;\V_{\text{div}})}\right.\nonumber\\&\qquad \left.+\|\doublewidetilde{\varphi} \|_{\mathrm{L}^{\infty}(0,T;\mathrm{H})}\sum_{k=1}^{p-2}\|\doublewidetilde{\varphi}\|_{\mathrm{L}^{2k}(0,T;\mathrm{V})}^{k}+\|\doublewidetilde{\varphi} \|_{\mathrm{L}^{2}(0,T;\mathrm{H}^2)}\sum_{k=1}^{p-2}\|\doublewidetilde{\varphi}\|_{\mathrm{L}^{\infty}((0,T)\times\Omega)}^{k}\right),\label{l} 
	\end{align} 
	for $p\geq 3$ and since $\F''$ has a polynomial growth. Further by using Theorems \ref{unique}, \ref{weakstrong} and Remark \ref{imp}, imply that
	\begin{align}
	\|\y\|_{\mathrm{L}^{2}(0,T;\V_{\text{div}})}&\leq C	\|\mathbf{k}(x,t) \|_{\mathrm{L}^{2}(0,T;\V_{\text{div}}')}  \leq C\lambda^2,\label{y}  \\
	\|\varrho\|_{\mathrm{L}^{2}(0,T;\mathrm{H})}  & \leq C	\|l(x,t)\|_{\mathrm{L}^{2}(0,T;\mathrm{H}^{-2})} \leq C\sum_{k=1}^{p-2}\lambda \lambda^{k}.\label{r}
	\end{align}

	In order to get the required bound in \eqref{k}, we take the inner product  of $\mathbf{k}(x,t)$ with $\y$ to obtain 
	\begin{align}
	\langle\mathbf{k}(x,t),\y\rangle &=- \langle(\doublewidetilde{\u} \cdot \nabla )\doublewidetilde{\u},\y\rangle  - \frac{1}{2}(\nabla a\doublewidetilde{\varphi}^2,\y) - ((\J\ast \doublewidetilde{\varphi}) \nabla \doublewidetilde{\varphi},\y) - (\nabla \uppi_\y,\y), \label{4.17} 
	\end{align}
	Using devergence free condition we get $(\nabla \uppi_\y,\y)=0$. Further using an integration by parts, divergence free condition and H\"older's inequality, we know that 
	\begin{align*}
	|\langle(\doublewidetilde{\u} \cdot \nabla )\doublewidetilde{\u},\y\rangle|\leq \|\doublewidetilde{\u}\|_{\mathbb{L}^4}^2\|\nabla\y\|\leq \|\doublewidetilde{\u}\|_{\mathbb{L}^4}^2\|\y\|_{\V_{\text{div}}},
	\end{align*}
	so that by using the Ladyzhenskaya inequality, we get 
	\begin{align*}
	\int_0^T\|(\doublewidetilde{\u}(t) \cdot \nabla )\doublewidetilde{\u}(t)\|_{\V_{\text{div}}'}^2\d t&\leq \int_0^T\|\doublewidetilde{\u}(t)\|_{\mathbb{L}^4}^4\d t\leq 2\sup_{t\in[0,T]}\|\doublewidetilde{\u}(t)\|^2\int_0^T\|\nabla\doublewidetilde{\u}(t)\|^2\d t.
	\end{align*}
	Thus, we obtain 
	\begin{align}\label{4.22a}
	\|(\doublewidetilde{\u}\cdot\nabla)\doublewidetilde{\u}\|_{\mathrm{L}^2(0,T;\V_{\text{div}}')}\leq \sqrt{2}\|\doublewidetilde{\u}\|_{\mathrm{L}^{\infty}(0,T;\G_{\text{div}})}\|\doublewidetilde{\u}\|_{\mathrm{L}^2(0,T;\V_{\text{div}})}.
	\end{align}
	Using H\"older's inequality and Poincar\'e inequality, we also have
	\begin{align*}
	\left|\frac{1}{2}(\nabla a \doublewidetilde{\varphi}^2,\y)\right| &\leq \frac{1}{2}\|\nabla a\|_{\mathbb{L}^{\infty}}  \|\doublewidetilde{\varphi}\|_{\mathrm{L}^4}^2 \| \y\|\leq \frac{1}{2\sqrt{\lambda_1}}\|\nabla a\|_{\mathbb{L}^{\infty}}  \|\doublewidetilde{\varphi}\|_{\mathrm{L}^4}^2 \| \y\|_{\V_{\text{div}}} . 
	\end{align*}
	Hence, we get 
	\begin{align*}
	\int_0^T\left\|\frac{\nabla a}{2} \doublewidetilde{\varphi}^2(t)\right\|_{\V_{\text{div}}'}^2\d t&\leq \frac{1}{4\lambda_1}\|\nabla a\|_{\mathbb{L}^{\infty}}^2\int_0^T\|\doublewidetilde{\varphi}(t)\|_{\mathrm{L}^4}^4\d t\no\\&\leq C\sup_{t\in[0,T]}\|\doublewidetilde{\varphi}(t)\|^2\int_0^T\|\nabla\doublewidetilde{\varphi}(t)\|^2\d t. 
	\end{align*}
	Thus, it follows that
	\begin{align}\label{4.23a}
	\left\|\frac{\nabla a}{2} \doublewidetilde{\varphi}^2\right\|_{\mathrm{L}^2(0,T;\V_{\text{div}}')} &\leq C\|\doublewidetilde{\varphi} \|_{\mathrm{L}^{\infty}(0,T;\mathrm{H})}\|\doublewidetilde{\varphi}\|_{\mathrm{L}^2(0,T;\mathrm{V})},
	\end{align}
	where $C=\frac{1}{\sqrt{2\lambda_1}}\|\nabla a\|_{\mathbb{L}^{\infty}}$. Once again using an integration by parts, H\"older's inequality and Poincar\'e inequality, we obtain 
	\begin{align*}
	|	((\J\ast \doublewidetilde{\varphi}) \nabla \doublewidetilde{\varphi},\y) |\leq \| \nabla \J\|_{\mathbb{L}^1}\|\doublewidetilde{\varphi}\|_{\mathrm{L}^4}^2   \|\y\|\leq \frac{1}{\sqrt{\lambda_1}}\| \nabla \J\|_{\mathbb{L}^1}\|\doublewidetilde{\varphi}\|_{\mathrm{L}^4}^2   \|\y\|_{\V_{\text{div}}}.
	\end{align*}
	An estimate similar to (\ref{4.23a}) and \eqref{Estimate J} yields 	
	\begin{align}\label{4.24a}
	\|(\J\ast \doublewidetilde{\varphi}) \nabla \doublewidetilde{\varphi}\|_{\mathrm{L}^2(0,T;\V_{\text{div}}')} &\leq C\|\doublewidetilde{\varphi} \|_{\mathrm{L}^{\infty}(0,T;\mathrm{H})}\|\doublewidetilde{\varphi}\|_{\mathrm{L}^2(0,T;\mathrm{V})},
	\end{align}
	where $C=\frac{1}{\sqrt{\lambda_1}}\|\nabla\mathrm{J}\|_{\mathbb{L}^1}$. Combining (\ref{4.22a})-(\ref{4.24a}), we finally obtain (\ref{k}).  
	 
	 Now, to get the required bound in \eqref{l}, we take inner product of $\mathcal{B}^{-1}(\varrho - \overline{\varrho}) $ to obtain 
		\begin{align}
	&( l(x,t)-\overline{l(x,t)},\mathcal{B}^{-1}(\varrho - \overline{\varrho}) ) = - ( \doublewidetilde{\u} \cdot \nabla\doublewidetilde{\varphi},\mathcal{B}^{-1}(\varrho - \overline{\varrho}) ) \nonumber\\&\quad  + \frac{1}{2}(\Delta (\F'''(\varphi_{\U^*} + \theta \doublewidetilde{\varphi})\doublewidetilde{\varphi}^2),\mathcal{B}^{-1}(\varrho - \overline{\varrho}))- \frac{1}{2}(\overline{\Delta (\F'''(\varphi_{\U^*} + \theta \doublewidetilde{\varphi})\doublewidetilde{\varphi}^2)},\mathcal{B}^{-1}(\varrho - \overline{\varrho})).  \label{4.18} 
	\end{align}
Recalling, $(\overline{\Delta (\F'''(\varphi_{\U^*} + \theta \doublewidetilde{\varphi})\doublewidetilde{\varphi}^2)},\mathcal{B}^{-1}(\varrho - \overline{\varrho}))=|\Omega|\overline{\Delta (\F'''(\varphi_{\U^*} + \theta \doublewidetilde{\varphi})\doublewidetilde{\varphi}^2)}\overline{\mathcal{B}^{-1}(\varrho - \overline{\varrho})}$ and $|\overline{f}|\leq \frac{1}{|\Omega|^{1/2}}\|f\|$. 	It should be noted that $\overline{\doublewidetilde{\u} \cdot \nabla \doublewidetilde{\varphi}}=0$, since an integration by parts, $\doublewidetilde{\u}\big|_{\partial\Omega}=0$ and the divergence free condition of $\doublewidetilde{\u}$  yields 
	\begin{align}\label{4.67}
	\overline{\doublewidetilde{\u} \cdot \nabla \doublewidetilde{\varphi}}&=\frac{1}{|\Omega|}\sum_{i=1}^2\int_{\Omega}\doublewidetilde{\u}_i(x)\frac{\partial \doublewidetilde{\varphi}(x)}{\partial x_i}\d x\nonumber\\&=\frac{1}{|\Omega|}\sum_{i=1}^2\left[\int_{\partial\Omega}\doublewidetilde{\varphi}(x)\doublewidetilde{\u}_i(x)\n_i(x)\d x-\int_{\Omega}\frac{\partial\doublewidetilde{\u}_i(x)}{\partial x_i}\doublewidetilde{\varphi}(x)\d x\right]=0.
	\end{align}
	We estimate the  first term in the right hand side of \eqref{4.18} using an integration by parts and H\"older's inequality as
	\begin{align}\label{4.54}
	|(\doublewidetilde{\u} \cdot \nabla \doublewidetilde{\varphi},\mathcal{B}^{-1}(\varrho - \overline{\varrho}))| = |(\doublewidetilde{\u} \cdot \nabla \mathcal{B}^{-1}(\varrho - \overline{\varrho}), \doublewidetilde{\varphi})| \leq \| \doublewidetilde{\u}\|_{\mathbb{L}^4}  \|\mathcal{B}^{-1/2}(\varrho - \overline{\varrho})\|\|\doublewidetilde{\varphi}\|_{\mathrm{L}^4}.
	\end{align}
	Since, \eqref{4.67} is true, from \eqref{bes1}, it is immediate that  $(\doublewidetilde{\u} \cdot \nabla \doublewidetilde{\varphi},\mathcal{B}^{-1}(\varrho - \overline{\varrho})) =(\mathcal{B}^{-1/2} (\doublewidetilde{\u} \cdot \nabla \doublewidetilde{\varphi}),\mathcal{B}^{-1/2}(\varrho - \overline{\varrho}))$. Thus, we have 
	\begin{align}\label{4.72a}
	\|\mathcal{B}^{-1/2} (\doublewidetilde{\u} \cdot \nabla \doublewidetilde{\varphi})\|=\sup_{\|\mathcal{B}^{-1/2}(\varrho - \overline{\varrho})\|=1}|(\mathcal{B}^{-1/2} (\doublewidetilde{\u} \cdot \nabla \doublewidetilde{\varphi}),\mathcal{B}^{-1/2}(\varrho - \overline{\varrho}))|\leq  \| \doublewidetilde{\u}\|_{\mathbb{L}^4}  \|\doublewidetilde{\varphi}\|_{\mathrm{L}^4}.
	\end{align}
	Also, using  \eqref{bes1}, we get $( l(x,t)-\overline{l(x,t)},\mathcal{B}^{-1}(\varrho - \overline{\varrho}) ) =(\mathcal{B}^{-1} (l(x,t)-\overline{l(x,t)}),\varrho - \overline{\varrho} ) $ and 
	\begin{align*}
	&\|\mathcal{B}^{-1} (l(x,t)-\overline{l(x,t)})\|=\sup_{\|\varrho - \overline{\varrho}\|=1}|(\mathcal{B}^{-1} (l(x,t)-\overline{l(x,t)}),\varrho - \overline{\varrho} )|\nonumber\\&\leq \sup_{\|\varrho - \overline{\varrho}\|=1}\bigg[ \| \doublewidetilde{\u}\|_{\mathbb{L}^4} \|\doublewidetilde{\varphi}\|_{\mathrm{L}^4} \|\mathcal{B}^{-1/2}(\varrho - \overline{\varrho})\|+|(\F'''(\varphi_{\U^*} + \theta \doublewidetilde{\varphi})\doublewidetilde{\varphi}^2,\varrho - \overline{\varrho})|\nonumber\\&\qquad+ \|\Delta (\F'''(\varphi_{\U^*} + \theta \doublewidetilde{\varphi})\doublewidetilde{\varphi}^2)\|\|\mathcal{B}^{-1}(\varrho - \overline{\varrho})\|\bigg]\nonumber\\&\leq  \| \doublewidetilde{\u}\|_{\mathbb{L}^4} \|\doublewidetilde{\varphi}\|_{\mathrm{L}^4} +\|\F'''(\varphi_{\U^*} + \theta \doublewidetilde{\varphi})\doublewidetilde{\varphi}^2\|+ \|\Delta (\F'''(\varphi_{\U^*} + \theta \doublewidetilde{\varphi})\doublewidetilde{\varphi}^2)\|,
	\end{align*}
	where we used \eqref{bes}. Using \eqref{4.72a}, we estimate $\|\doublewidetilde{\u} \cdot \nabla \doublewidetilde{\varphi}\|_{\mathrm{L}^2(0, T; \mathrm{V}^{'})}$  as 
	\begin{align*}
	&\int_0^T\|\mathcal{B}^{-1/2}(\doublewidetilde{\u} \cdot \nabla \doublewidetilde{\varphi}(t))\|^2\d t \leq C \int_0^T\|\doublewidetilde{\u}(t)\|_{\mathbb{L}^4}^2 \|\doublewidetilde{\varphi}(t)\|_{\mathrm{L}^4}^2 \d t \\
	&\leq 2C\left(\int_0^T\|\doublewidetilde{\u}(t)\|^2\|\nabla\doublewidetilde{\u}(t)\|^2\d t\right)^{1/2}\left(\int_0^T\|\doublewidetilde{\varphi}(t)\|^2\|\nabla\doublewidetilde{\varphi}(t)\|^2\d t\right)^{1/2}\\&\leq C \sup_{t\in[0,T]}\|\doublewidetilde{\u}(t)\|^2\left(\int_0^T\|\nabla\doublewidetilde{\u}(t)\|^2\d t\right)+C\sup_{t\in[0,T]}\|\doublewidetilde{\varphi}(t)\|^2\left(\int_0^T\|\nabla\doublewidetilde{\varphi}(t)\|^2\d t\right),
	\end{align*}
	where we used the Ladyzhenskaya, H\"older, Young's inequalities. 
	Hence, we have  
	\begin{align}\label{est}
	\|\doublewidetilde{\u} \cdot \nabla \doublewidetilde{\varphi} \|_{\mathrm{L}^{2}(0,T;\mathrm{V}')} &\leq C\left( \|\doublewidetilde{\u}\|_{\mathrm{L}^{\infty}(0,T;\G_{\text{div}})}\|\doublewidetilde{\u}\|_{\mathrm{L}^2(0,T;\V_{\text{div}})}+\|\doublewidetilde{\varphi} \|_{\mathrm{L}^{\infty}(0,T;\mathrm{H})}\|\doublewidetilde{\varphi}\|_{\mathrm{L}^2(0,T;\mathrm{V})}\right).
	\end{align}	
	In order to estimate the second term in \eqref{4.18}, for simplicity, we take $\F(s)=(s^2-1)^2$ and  estimate $\frac{1}{2}(\Delta (\F'''(\varphi_{\U^*} + \theta \doublewidetilde{\varphi})\doublewidetilde{\varphi}^2),\mathcal{B}^{-1}(\varrho - \overline{\varrho}))$ using \eqref{bes}, H\"older's and Gagliardo-Nirenberg inequalities as:
	\begin{align}\label{4.72}
	\frac{1}{2}|(\Delta(\F'''(\varphi_{\U^*} + \theta \doublewidetilde{\varphi})\doublewidetilde{\varphi}^2),\mathcal{B}^{-1}(\varrho-\overline{\varrho}))|&=  \frac{1}{2}|(\F'''(\varphi_{\U^*} + \theta \doublewidetilde{\varphi})\doublewidetilde{\varphi}^2,\varrho-\overline{\varrho})|\nonumber \\  &\leq \frac{1}{2}\|(\F'''(\varphi_{\U^*} + \theta \doublewidetilde{\varphi})\doublewidetilde{\varphi}^2\|\|\varrho-\overline{\varrho}\|\nonumber\\&\leq 12\|(\varphi_{\U^*} + \theta \doublewidetilde{\varphi})\doublewidetilde{\varphi}^2\|\|\varrho-\overline{\varrho}\|\no\\&\leq 12\left(\|\varphi_{\U^*} \doublewidetilde{\varphi}^2\|+\|\doublewidetilde{\varphi}^3\|\right)\|\varrho-\overline{\varrho}\|\nonumber\\&\leq 12\left(\|\varphi_{\U^*}\|_{\mathrm{L}^{\infty}}\|\doublewidetilde{\varphi}\|_{\mathrm{L}^4}^2+\|\doublewidetilde{\varphi}\|_{\mathrm{L}^6}^3\right)\|\varrho-\overline{\varrho}\|\nonumber\\&\leq C\left(\|\varphi_{\U^*}\|_{\mathrm{L}^{\infty}}\|\doublewidetilde{\varphi}\|\|\nabla\doublewidetilde{\varphi}\|+\|\nabla\doublewidetilde{\varphi}\|^2\|\doublewidetilde{\varphi}\|\right)\|\varrho-\overline{\varrho}\|.
	\end{align}
	Now from \eqref{4.72}, we infer that 
	\begin{align*}
	\frac{1}{2}\|\F'''(\varphi_{\U^*} + \theta \doublewidetilde{\varphi})\doublewidetilde{\varphi}^2\|\leq C\left(\|\varphi_{\U^*}\|_{\mathrm{L}^{\infty}}\|\doublewidetilde{\varphi}\|\|\nabla\doublewidetilde{\varphi}\|+\|\nabla\doublewidetilde{\varphi}\|^2\|\doublewidetilde{\varphi}\|\right).
	\end{align*}
	Thus, combining the above two estimates, we have 
	\begin{align}\label{55}
	&\left\|\frac{1}{2}(\F'''(\varphi_{\U^*} + \theta \doublewidetilde{\varphi})\doublewidetilde{\varphi}^2)\right\|_{\mathrm{L}^2(0,T;\mathrm{H})}\nonumber\\&\leq C\left(\|\varphi_{\U^*}\|_{\mathrm{L}^{\infty}((0,T)\times\Omega)}\|\doublewidetilde{\varphi}\|_{\mathrm{L}^{\infty}(0,T;\mathrm{H})}\|\doublewidetilde{\varphi}\|_{\mathrm{L}^2(0,T;\mathrm{V})}+\|\doublewidetilde{\varphi}\|_{\mathrm{L}^{\infty}(0,T;\mathrm{H})}\|\doublewidetilde{\varphi}\|_{\mathrm{L}^4(0,T;\mathrm{V})}^2\right).
	\end{align}
	Using the fact that  $\F(\cdot)$ has polynomial growth and  Gagliardo-Nirenberg inequality, the general case of polynomial of order $p$ can be obtained in a similar way as 
	\begin{align}\label{4.74}
	&\left\|\frac{1}{2}|(\F'''(\varphi_{\U^*} + \theta \doublewidetilde{\varphi})\doublewidetilde{\varphi}^2)\right\|_{\mathrm{L}^2(0,T;\mathrm{H})}\no\\&\quad\leq C\left(\|\varphi_{\U^*}\|_{\mathrm{L}^{\infty}((0,T)\times\Omega)}\right)\|\doublewidetilde{\varphi}\|_{\mathrm{L}^{\infty}(0,T;\mathrm{H})}\sum_{k=1}^{p-2}\|\doublewidetilde{\varphi}\|_{\mathrm{L}^{2k}(0,T;\mathrm{V})}^{k},
	\end{align}
	for $p\geq 3$. Now, we estimate the final term in \eqref{4.18} as 
	\begin{align}\label{4.77}
	\frac{1}{2}(\overline{\Delta (\F'''(\varphi_{\U^*} + \theta \doublewidetilde{\varphi})\doublewidetilde{\varphi}^2)},\mathcal{B}^{-1}(\varrho - \overline{\varrho}))&\leq\frac{1}{2} \|\Delta (\F'''(\varphi_{\U^*} + \theta \doublewidetilde{\varphi})\doublewidetilde{\varphi}^2)\|\|\mathcal{B}^{-1}(\varrho - \overline{\varrho})\|\nonumber\\&\leq \frac{1}{2} \|\Delta (\F'''(\varphi_{\U^*} + \theta \doublewidetilde{\varphi})\doublewidetilde{\varphi}^2)\|\|\varrho - \overline{\varrho}\|.
	\end{align}
	As before, for simplicity, we take $\F(s)=(s^2-1)^2$ and estimate the first term in the right hand side of \eqref{4.77} as 
	\begin{align*}
	&\frac{1}{2}\|\Delta (\F'''(\varphi_{\U^*} + \theta \doublewidetilde{\varphi})\doublewidetilde{\varphi}^2)\|\leq 12\|\Delta((\varphi_{\U^*} + \theta \doublewidetilde{\varphi})\doublewidetilde{\varphi}^2)\|\nonumber\\&\leq 12\Big((2\|\varphi_{\U^*}\|_{\mathrm{L}^{\infty}}\|\doublewidetilde{\varphi}\|_{\mathrm{L}^{\infty}}+3\|\doublewidetilde{\varphi}\|_{\mathrm{L}^{\infty}}^2)\|\Delta\doublewidetilde{\varphi}\| +\|\doublewidetilde{\varphi}\|_{\mathrm{L}^{\infty}}^2\|\Delta\varphi_{\U^*}\|\nonumber\\&\quad +(2\|\varphi_{\U^*}\|_{\mathrm{L}^{\infty}}+6\|\doublewidetilde{\varphi}\|_{\mathrm{L}^{\infty}})\|\nabla\doublewidetilde{\varphi}\|^2+4\|\nabla\varphi_{\U^*}\|\|\varphi\|_{\mathrm{L}^{\infty}}\|\nabla\doublewidetilde{\varphi}\|\Big).
	\end{align*}
	Thus, we have, 
	\begin{align}\label{4.78}
	&\left\|\frac{1}{2}\Delta(\F'''(\varphi_{\U^*} + \theta \doublewidetilde{\varphi})\doublewidetilde{\varphi}^2)\right\|_{\mathrm{L}^2(0,T;\mathrm{H})}\nonumber\\&\leq C\bigg[\Big(\|\varphi_{\U^*}\|_{\mathrm{L}^{\infty}((0,T)\times\Omega)}\|\doublewidetilde{\varphi}\|_{\mathrm{L}^{\infty}((0,T)\times\Omega)}+\|\doublewidetilde{\varphi}\|_{\mathrm{L}^{\infty}((0,T)\times\Omega)}^2\Big)\|\doublewidetilde{\varphi}\|_{\mathrm{L}^{2}(0,T;\mathrm{H}^2)}\nonumber\\&\quad+\Big(\|\varphi_{\U^*}\|_{\mathrm{L}^{\infty}((0,T)\times\Omega)}+\|\doublewidetilde{\varphi}\|_{\mathrm{L}^{\infty}((0,T)\times\Omega)}\Big)\|\doublewidetilde{\varphi}\|_{\mathrm{L}^4(0,T;\mathrm{V})}^2\nonumber\\&\quad +\|\varphi_{\U^*}\|_{\mathrm{L}^4(0,T;\mathrm{V})}\|\doublewidetilde{\varphi}\|_{\mathrm{L}^{\infty}((0,T)\times\Omega)}\|\doublewidetilde{\varphi}\|_{\mathrm{L}^4(0,T;\mathrm{V})} +\|\varphi_{\U^*}\|_{\mathrm{L}^{2}(0,T;\mathrm{H}^2)}\|\doublewidetilde{\varphi}\|_{\mathrm{L}^{\infty}((0,T)\times\Omega)}^2\nonumber\\&\leq C\lambda(\lambda+\lambda^2).
	\end{align}

	In gerneral, since $\F(\cdot)$ has a polynomial growth, a similar kind of estimate holds true. Combining \eqref{est}- \eqref{4.78}, we get \eqref{l}. Also, the validity of (\ref{y}) and (\ref{r}) is immediate using Theorems \ref{unique}, \ref{weakstrong} and \ref{thm2.16}. Hence by taking $\lambda \rightarrow 0$, we finally arrive at (\ref{conv1}) and (\ref{conv2}). 

It completes the proof.
\end{proof}

\begin{remark}
	If $(\u_0, \varphi_0)$ satisfies (\ref{initial}) and $\F(\cdot)$ satisfies \eqref{fes}, the mapping $\U \mapsto (\u_{\U},\varphi_{\U})$  from $\mathscr{U}_{\text{ad}}$ into  $\left(\mathrm{C}([0,T];\G_{\text{div}})\cap\mathrm{L}^{2}(0,T;\V_{\text{div}})\right) \times \left(\mathrm{C}([0,T];\mathrm{H})\cap\mathrm{L}^{2}(0,T;\mathrm{V})\right)$ is G\^ateaux differentiable. In this case, \eqref{conv2} is replaced by 
	\begin{align*}
	\lim_{\lambda\to0}\frac{\|\varphi_{\U^*+\lambda \U} - \varphi_{\U^*}-\lambda\psi\|_{\mathrm{L}^2(0,T;\mathrm{V})}}{|\lambda|}=0.
	\end{align*} This strong G\^ateaux differentiability of the mapping $\U \mapsto (\u_{\U},\varphi_{\U})$ is similar to that of Theorem 3.4, \cite{ControlCHNS}.  In order to establish this, we take the inner product of $\widetilde{\mu}_{\varrho} $ with $l(x,t)$ in \eqref{4.58} to obtain  
	\begin{align*}
	(l(x,t),\widetilde{\mu}_{\varrho})=- (\doublewidetilde{\u} \cdot \nabla\doublewidetilde{\varphi},\widetilde{\mu}_{\varrho})+\frac{1}{2}( \Delta (\F'''(\varphi_{\U^*} + \theta \doublewidetilde{\varphi})\doublewidetilde{\varphi}^2),\widetilde{\mu}_{\varrho}).
	\end{align*}
	Now an integration by parts and H\"older's inequality yield
	\begin{align*}
	- (\doublewidetilde{\u} \cdot \nabla\doublewidetilde{\varphi},\widetilde{\mu}_{\varrho})=(\doublewidetilde{\u}\cdot\nabla\widetilde{\mu}_{\varrho},\doublewidetilde{\varphi})\leq \|\doublewidetilde{\u}\|_{\mathbb{L}^4}\|\nabla\widetilde{\mu}_{\varrho}\|\|\doublewidetilde{\varphi}\|_{\L^4}\leq \|\doublewidetilde{\u}\|_{\mathbb{L}^4}\|\doublewidetilde{\varphi}\|_{\L^4}\|\widetilde{\mu}_{\varrho}\|_{\mathrm{V}}.
	\end{align*}
	A calculation similar to \eqref{4.78} gives
	\begin{align*}
	\frac{1}{2}|( \Delta (\F'''(\varphi_{\U^*} + \theta \doublewidetilde{\varphi})\doublewidetilde{\varphi}^2),\widetilde{\mu}_{\varrho})|&\leq \| \Delta (\F'''(\varphi_{\U^*} + \theta \doublewidetilde{\varphi})\doublewidetilde{\varphi}^2)\|\|\widetilde{\mu}_{\varrho}\|\nonumber\\&\leq C\left( \|\doublewidetilde{\varphi} \|\sum_{k=1}^{p-2}\|\doublewidetilde{\varphi}\|_{{\mathrm{V}}}^{k}+\|\doublewidetilde{\varphi} \|_{\mathrm{H}^2}\sum_{k=1}^{p-2}\|\doublewidetilde{\varphi}\|_{\mathrm{L}^{\infty}}^{k}\right)\|\widetilde{\mu}_{\varrho}\|_{\mathrm{V}},
	\end{align*}
	for $p\geq 3$. Thus, we have 
	\begin{align*}
	\|l(x,t)\|_{\mathrm{V}'}\leq C\left(\|\doublewidetilde{\u}\|_{\mathbb{L}^4}\|\doublewidetilde{\varphi}\|_{\L^4}+ \|\doublewidetilde{\varphi} \|\sum_{k=1}^{p-2}\|\doublewidetilde{\varphi}\|_{{\mathrm{V}}}^{k}+\|\doublewidetilde{\varphi} \|_{\mathrm{H}^2}\sum_{k=1}^{p-2}\|\doublewidetilde{\varphi}\|_{\mathrm{L}^{\infty}}^{k}\right),
	\end{align*}
	and 
	\begin{align*}
	\|\widetilde{\mu}_{\varrho}\|_{\mathrm{L}^{2}(0,T;\mathrm{V})} & \leq C	\|l(x,t)\|_{\mathrm{L}^{2}(0,T;\mathrm{V}')} \\ &\leq C\left( \|\doublewidetilde{\u}\|_{\mathrm{L}^{\infty}(0,T;\G_{\text{div}})}\|\doublewidetilde{\u}\|_{\mathrm{L}^2(0,T;\V_{\text{div}})}+\|\doublewidetilde{\varphi} \|_{\mathrm{L}^{\infty}(0,T;\mathrm{H})}\sum_{k=1}^{p-2}\|\doublewidetilde{\varphi}\|_{\mathrm{L}^{2k}(0,T;\mathrm{V})}^{k}\right.\nonumber\\&\qquad \left.+\|\doublewidetilde{\varphi} \|_{\mathrm{L}^{2}(0,T;\mathrm{H}^2)}\sum_{k=1}^{p-2}\|\doublewidetilde{\varphi}\|_{\mathrm{L}^{\infty}((0,T)\times\Omega)}^{k}\right)\leq C\sum_{k=1}^{p-2}\lambda \lambda^k.
	\end{align*}
	It can be shown that 
	\begin{align*}
	\|\varrho\|_{\mathrm{L}^{2}(0,T;\mathrm{V})}\leq \|\widetilde{\mu}_{\varrho}\|_{\mathrm{L}^{2}(0,T;\mathrm{V})} \leq  C\sum_{k=1}^{p-2}\lambda \lambda^k,
	\end{align*}
	which gives G\^ateaux differentiability of  $\varphi_{\U}$  in the strong sense.
\end{remark}

\section{Second order necessary and sufficient optimality condition}\label{se5}
	In this section we derive the second order necessary and sufficient optimality condition for the optimal control problem \eqref{control problem}. 
	
	Let $(\widehat{\u},\widehat{\varphi},\widehat{\U})$ be an arbitrary feasible triplet for the optimal control problem \eqref{control problem}, we set
	\begin{align}\label{def Q}
	\mathcal{Q}_{\widehat{\u},\widehat{\varphi},\widehat{\U}}= \{(\u,\varphi,\U) \in \mathcal{A}_{\text{ad}} \} - \{(\widehat{\u},\widehat{\varphi},\widehat{\U})\},
	\end{align}
	which denotes the differences of all feasible triplets for the problem \eqref{control problem} corresponding to $(\widehat{\u},\widehat{\varphi},\widehat{\U})$.
	\begin{theorem}[Necessary condition]\label{necessary}
		Let  $(\u^*,\varphi^*,\U^*)$ be an optimal triplet for the problem \eqref{control problem} and the adjoint variables $(\p,\eta)$ satisfies the adjoint system \eqref{adj}. Then for any $(\u,\varphi,\U) \in \mathcal{Q}_{\u^*,\varphi^*,\U^*}$, there exist $0\leq \theta\leq1$ such that
		\begin{align}\label{3.14}
		& \int_0^T \|\u(t)\|^2 \d t+  \int_0^T \|\varphi(t)\|^2 \d t+ \int_0^T\|\U(t)\|^2\d t +\int_0^T \big((\F'''(\varphi^*+ \theta \varphi)\varphi^2,\Delta\eta\big)  \nonumber \\ 
		&\quad-2\int_0^T ( (\u \cdot \nabla )\u+ \frac{\nabla a}{2} \varphi^2 +(\J\ast \varphi) \nabla \varphi - \U, \p ) \d t - 2\int_0^T (\u \cdot \nabla \varphi,\eta ) \d t  \geq 0.
		\end{align}
	\end{theorem}
	\begin{proof}
		For any $(\u,\varphi,\U)\in \mathcal{Q}_{\u^*,\varphi^*,\U^*},$ by  \eqref{def Q} there exist $(\z,\xi,\W) \in \mathcal{A}_{\text{ad}}$ such that $(\u,\varphi,\U)= (\z - \u^*,\xi - \varphi^*,\W - \U^*)$. So from \eqref{nonlin phi}-\eqref{initial conditions}, we can derive that $(\u,\varphi)$ satisfies the following system:
		\begin{eqnarray}\label{5.29}
		\left\{
		\begin{aligned}
		{\u}_t - \nu \Delta \u + (\u \cdot \nabla )\u^* &+ (\u^* \cdot \nabla )\u +(\J\ast \varphi) \nabla \varphi^*+ (\J\ast \varphi^*) \nabla \varphi+\nabla a \varphi \varphi^* + \nabla \widetilde{\uppi}_{\u} \\&= -(\u \cdot \nabla )\u- \frac{\nabla a}{2} \varphi^2 -(\J\ast \varphi) \nabla \varphi + \U, \ \text{ in } \ \Omega\times(0,T),\\
		{\varphi}_t + \u \cdot \nabla \varphi + \u^* \cdot\nabla \varphi + \u \cdot \nabla \varphi^* &= \Delta \widetilde{\mu}, \ \text{ in } \ \Omega\times(0,T), \\ 
		\widetilde{\mu} &= a \varphi - \J\ast \varphi +\F'(\varphi^*+  \varphi) - \F'(\varphi^*) ,\ \text{ in } \ \Omega\times(0,T),  \\
		\text{div }\u &= 0, \ \text{ in } \ \Omega\times(0,T), \\
		\frac{\partial \widetilde{\mu}}{\partial\mathbf{n}} &= 0, \ \u=0,  \ \text{on } \ \partial \Omega \times (0,T),\\
		\u(0) &= 0, \ \varphi(0) = 0, \ \text{ in } \ \Omega.
		\end{aligned}
		\right.
		\end{eqnarray}
		From now on we use the Taylor series expansion of $\mathrm{F}'$ around $\varphi^*$. There exists a $\theta$;  $0<\theta<1$ such that 
		\begin{align}\label{3.15}
		\F'(\varphi^*+  \varphi) - \F'(\varphi^*) =\F''(\varphi^*)  \varphi + \F'''(\varphi^*+ \theta \varphi)\frac{\varphi^2}{2}.
		\end{align}
		Taking inner product of \eqref{5.29} with $(\p,\eta)$, integrating over $[0,T]$ and then adding, we get
		\begin{align}
		&\int_0^T ({\u}_t - \nu \Delta \u + (\u \cdot \nabla )\u^* + (\u^* \cdot \nabla )\u +(\J\ast \varphi) \nabla \varphi^*+ (\J\ast \varphi^*) \nabla \varphi+\nabla a \varphi \varphi^* + \nabla \widetilde{\uppi}_{\u},\p ) \d t \nonumber\\  & \quad+ \int_0^T ( (\u \cdot \nabla )\u+ \frac{\nabla a}{2} \varphi^2 +(\J\ast \varphi) \nabla \varphi - \U, \p ) \d t\nonumber\\&\quad+ \int_0^T ( {\varphi}_t + \u \cdot \nabla \varphi + \u^* \cdot\nabla \varphi + \u \cdot \nabla \varphi^* - \Delta \widetilde{\mu},\eta ) \d t =0 .
		\end{align}
		Using an integration by parts, we further get
		\begin{align}\label{3.18}
		&\int_0^T ( \u,-\p_t- \nu \Delta \p +(\p \cdot \nabla)\u^* + (\u^* \cdot \nabla)\p + \nabla \varphi^* \eta+\nabla q ) \d t\nonumber\\& + \int_0^T ( \varphi, -\eta_t  +\J \ast (\p \cdot \nabla \varphi^*) - (\nabla{\J} \ast  \varphi^*)\cdot \p + \nabla a\cdot \p \varphi^* - \u^* \cdot\nabla \eta)\d t\nonumber\\& + \int_0^T ( \varphi,  - a \Delta \eta + \J\ast \Delta \eta - \F''(\varphi^*)\Delta \eta )\d t+ \int_0^T (\u \cdot \nabla \varphi,\eta ) \d t\nonumber\\&+ \int_0^T ( (\u \cdot \nabla )\u+ \frac{\nabla a}{2} \varphi^2 +(\J\ast \varphi) \nabla \varphi - \U, \p ) \d t  -\int_0^T ((\F'''(\varphi^*+ \theta \varphi)\frac{\varphi^2}{2},\Delta\eta) \d t =0.
		\end{align}
		Since $(\u^*,\varphi^*,\U^*)$ is an optimal triplet, it satisfies the first order necessary conditions given in \eqref{nec}. This and the adjoint system \eqref{adj} implies 
		\begin{align}\label{5.32}
		& \int_0^T ( \u,\u^* -\u_d) \d t+ \int_0^T (\varphi,\varphi^*-\varphi_d ) \d t+ \int_0^T (\U,\U^*) \d t  \nonumber\\&
		\quad +\int_0^T ( (\u \cdot \nabla )\u+ \frac{\nabla a}{2} \varphi^2 +(\J\ast \varphi) \nabla \varphi , \p ) \d t + \int_0^T (\u \cdot \nabla \varphi,\eta ) \d t \nonumber \\
		& \quad-\int_0^T ((\F'''(\varphi^*+ \theta \varphi)\frac{\varphi^2}{2},\Delta\eta) \d t=0.
		\end{align}
		Since $(\u,\varphi,\U) \in \mathcal{Q}_{\u^*,\varphi^*,\U^*}$, by \eqref{def Q}, we have $(\u+\u^*,\varphi+\varphi^*,\U+\U^*)$ is a feasible triplet for the problem \eqref{control problem}. We obtain that
		\begin{align}
		&\mathcal{J}(\u+\u^*,\varphi+\varphi^*,\U+\U^*) - \mathcal{J}(\u^*,\varphi^*,\U^*) \nonumber\\&= \ \int_0^T \|(\u+\u^*-\u_d)(t)\|^2 \d t+  \int_0^T \|(\varphi+\varphi^*-\varphi_d)(t)\|^2 \d t+ \int_0^T\|(\U+\U^*)(t)\|^2\d t \nonumber \\
		&\quad  - \int_0^T \|(\u^*-\u_d)(t)\|^2 \d t-  \int_0^T \|(\varphi^*-\varphi_d)(t)\|^2 \d t- \int_0^T\|\U^*(t)\|^2\d t \geq 0 \nonumber \\
		& =\int_0^T \|\u(t)\|^2 \d t+  \int_0^T \|\varphi(t)\|^2 \d t+ \int_0^T\|\U(t)\|^2\d t \nonumber \\
		&\quad +2\int_0^T ( \u,\u^* -\u_d) \d t+2 \int_0^T (\varphi,\varphi^*-\varphi_d ) \d t+ 2\int_0^T (\U,\U^*) \d t \geq 0.
		\end{align}
		From \eqref{5.32}, it follows that
		\begin{align*}
	&	\int_0^T \|\u(t)\|^2 \d t+  \int_0^T \|\varphi(t)\|^2 \d t+ \int_0^T\|\U(t)\|^2\d t \nonumber \\
	&\quad -\int_0^T ( (\u \cdot \nabla )\u+ \frac{\nabla a}{2} \varphi^2 +(\J\ast \varphi) \nabla \varphi , \p ) \d t - \int_0^T (\u \cdot \nabla \varphi,\eta ) \d t \nonumber \\
		& \quad+\int_0^T ((\F'''(\varphi^*+ \theta \varphi)\frac{\varphi^2}{2},\Delta\eta) \d t \geq 0.
		\end{align*}
		which completes the proof.
	\end{proof}
	\begin{remark}\label{rem3.9}
		Note that we assumed Dirichlet boundary condition on $\eta$ in \eqref{adj}. One can perform integration by parts in (\ref{3.18}) and obtain  the term $(\F'''(\varphi^*+ \theta_4 \varphi)\frac{\varphi^2}{2},\Delta\eta)$, which is well-defined. 
	\end{remark}
	
	\begin{theorem}[Sufficient condition]\label{sufficient}
		Suppose that $(\u^*,\varphi^*,\U^*)$ be a feasible triplet for the problem \eqref{control problem}. Let us assume that the first order necessary condition holds true (see  \eqref{nec}), and for any $0\leq \theta\leq 1$ and $(\u,\varphi,\U) \in \mathcal{Q}_{\u^*,\varphi^*,\U^*},$ the following inequality holds:
		\begin{align}\label{3.21}
		&	\int_0^T \|\u(t)\|^2 \d t+  \int_0^T \|\varphi(t)\|^2 \d t+ \int_0^T\|\U(t)\|^2\d t \nonumber \\
	&\quad -\int_0^T ( (\u \cdot \nabla )\u+ \frac{\nabla a}{2} \varphi^2 +(\J\ast \varphi) \nabla \varphi , \p ) \d t - \int_0^T (\u \cdot \nabla \varphi,\eta ) \d t \nonumber \\
		& \quad+\int_0^T ((\F'''(\varphi^*+ \theta \varphi)\frac{\varphi^2}{2},\Delta\eta) \d t \geq 0.
		\end{align}
		Then $(\u^*,\varphi^*,\U^*)$ is an optimal triplet for the problem \eqref{control problem}.
	\end{theorem}
	\begin{proof}
		For any $(\z,\xi,\W) \in \mathcal{A}_{\text{ad}}$, we have by \eqref{def Q} that, $(\z - \u^*,\xi - \varphi^*,\W - \U^*) \in \mathcal{Q}_{\u^*,\varphi^*,\U^*}$ and  it satisfies:
		\begin{eqnarray}\label{5.34}
		\left\{
		\begin{aligned}
		({\z - \u^*})_t &- \nu \Delta (\z - \u^*) + ((\z - \u^*) \cdot \nabla )\u^* + (\u^* \cdot \nabla )(\z - \u^*) +(\J\ast (\xi - \varphi^*)) \nabla \varphi^*   \\
		&\quad+ (\J\ast \varphi^*) \nabla (\xi - \varphi^*) +\nabla a (\xi - \varphi^*) \varphi^*+ \nabla \widetilde{\uppi}_{(\z - \u^*)} \\&= -((\z - \u^*) \cdot \nabla )(\z - \u^*)- \frac{\nabla a}{2} (\xi - \varphi^*)^2-(\J\ast (\xi - \varphi^*)) \nabla (\xi - \varphi^*) \\&\quad+ \W -\U^*, \ \text{ in }\ \Omega\times(0,T),\\
		{(\xi - \varphi^*)}_t &+ (\z - \u^*) \cdot \nabla (\xi - \varphi^*) + \u^* \cdot\nabla (\xi - \varphi^*) + (\z - \u^*) \cdot \nabla \varphi^* \\&= \Delta \widetilde{\mu},\ \text{ in }\ \Omega\times(0,T), \\ 
		\widetilde{\mu} &= a (\xi - \varphi^*) - \J\ast (\xi - \varphi^*) +\mathrm{F}'(\xi)-\mathrm{F}'(\varphi^*),\\
		\text{div }(\z - \u^*) &= 0,\ \text{ in }\ \Omega\times(0,T), \\
		\frac{\partial \widetilde{\mu}}{\partial\mathbf{n}} &= 0, \ (\z - \u^*)=0,  \ \text{ on } \ \partial \Omega \times (0,T),\\
		(\z - \u^*)(0) &= 0, \ (\xi - \varphi^*)(0) = 0, \ \text{ in } \ \Omega.
		\end{aligned}
		\right.
		\end{eqnarray}
		As we argued in \eqref{3.15}, there exists a $0\leq \widetilde{\theta}\leq 1$ such that 
		\begin{align*}
		\mathrm{F}'(\xi)-\mathrm{F}'(\varphi^*)= \F''(\varphi^*) (\xi - \varphi^*) + \F'''(\varphi^*+\widetilde{\theta}(\xi - \varphi^*)) \frac{(\xi - \varphi^*)^2}{2}.
		\end{align*}
		Now multiplying \eqref{5.34} with $(\p,\eta)$, integrating over $[0,T]$ and then adding, we get  
		\begin{align}
		& \int_0^T ( \z-\u^*,-\p_t- \nu \Delta \p +(\p \cdot \nabla)\u^* + (\u^* \cdot \nabla)\p + \nabla \varphi^* \eta+\nabla q ) \d t \nonumber\\&
		\quad+ \int_0^T ( \xi-\varphi^*, -\eta_t   +\J \ast (\p \cdot \nabla \varphi^*) - (\nabla{\J} \ast  \varphi^*)\cdot \p + \nabla a\cdot \p \varphi^* - \u^* \cdot\nabla \eta - a \Delta \eta)\d t\nonumber\\&\quad  + \int_0^T ( \xi-\varphi^*,\J\ast \Delta \eta - \F''(\varphi^*)\Delta \eta) \d t 
		\nonumber\\&\quad +
		\int_0^T ( ((\z - \u^*) \cdot \nabla )(\z - \u^*) + \frac{\nabla a}{2} (\xi - \varphi^*)^2+(\J\ast (\xi - \varphi^*)) \nabla (\xi - \varphi^*) -\W + \U^*, \p ) \d t \nonumber \\
		&\quad + \int_0^T ((\z - \u^*) \cdot \nabla (\xi - \varphi^*),\eta) \d t  - \int_0^T((\F'''(\varphi^*+\widetilde{\theta}_4(\xi - \varphi^*)) \frac{(\xi - \varphi^*)^2}{2}),\Delta\eta)\d t =0,
		\end{align}
		where we also performed an integration by parts. Since $(\u^*,\varphi^*,\U^*)$ satisfies the first order necessary condition, i.e., $-\p(t) = \U^*(t),  \text{ a.e. } t \in [0,T]$  and using the adjoint system \eqref{adj}, we further get
		\begin{align}\label{5.35}
		& \int_0^T ( \z-\u^*,\u^*- \u_d) ) \d t + \int_0^T ( \xi-\varphi^*, \varphi^* -\varphi_d ) \d t + \int_0^T ( \U^*, \W - \U^* ) \d t \nonumber\\
		&\quad+\int_0^T ( ((\z - \u^*) \cdot \nabla )(\z - \u^*) + \frac{\nabla a}{2} (\xi - \varphi^*)^2+(\J\ast (\xi - \varphi^*)) \nabla (\xi - \varphi^*) , \p ) \d t \nonumber \\
		&\quad + \int_0^T ((\z - \u^*) \cdot \nabla (\xi - \varphi^*),\eta ) \d t -\int_0^T((\F'''(\varphi^*+\widetilde{\theta}(\xi - \varphi^*)) \frac{(\xi - \varphi^*)^2}{2}),\Delta\eta)\d t= 0.
		\end{align} 
		We know that
		\begin{align}\label{3.25}
		&\int_0^T \|(\z - \u_d)(t)\|^2 \d t+  \int_0^T \|(\xi-\varphi_d)(t)\|^2 \d t+ \int_0^T\|\W(t)\|^2\d t \nonumber \\
		&\quad  - \int_0^T \|(\u^*-\u_d)(t)\|^2 \d t-  \int_0^T \|(\varphi^*-\varphi_d)(t)\|^2 \d t- \int_0^T\|\U^*(t)\|^2\d t \nonumber \\
		&\quad = \int_0^T \|\z(t)\|^2 \d t+  \int_0^T \|\xi(t)\|^2 \d t+ \int_0^T\|\W(t)\|^2\d t \nonumber \\
		&\quad +\int_0^T ( \z-\u^*,\u^*- \u_d) ) \d t + \int_0^T ( \xi-\varphi^*, \varphi^* -\varphi_d ) \d t + \int_0^T ( \U^*, \W - \U^* ) \d t \nonumber\\
		&= \int_0^T \|\z(t)\|^2 \d t+  \int_0^T \|\xi(t)\|^2 \d t+ \int_0^T\|\W(t)\|^2\d t \nonumber \\
		&\quad -\int_0^T ( ((\z - \u^*) \cdot \nabla )(\z - \u^*) + \frac{\nabla a}{2} (\xi - \varphi^*)^2+(\J\ast (\xi - \varphi^*)) \nabla (\xi - \varphi^*) , \p ) \d t \nonumber \\
		&\quad - \int_0^T ((\z - \u^*) \cdot \nabla (\xi - \varphi^*),\eta ) \d t +\int_0^T((\F'''(\varphi^*+\widetilde{\theta}(\xi - \varphi^*)) \frac{(\xi - \varphi^*)^2}{2}),\Delta\eta)\d t
		\end{align}
		 Using \eqref{3.21}, we obtain for any $(\z,\xi,\W) \in \mathcal{A}_{\text{ad}}$, the following inequality holds:
		\begin{align*}
		&\int_0^T \|(\z - \u_d)(t)\|^2 \d t+  \int_0^T \|(\xi-\varphi_d)(t)\|^2 \d t+ \int_0^T\|\W(t)\|^2\d t  \geq \\
		&\quad   \int_0^T \|(\u^*-\u_d)(t)\|^2 \d t+  \int_0^T \|(\varphi^*-\varphi_d)(t)\|^2 \d t+ \int_0^T\|\U^*(t)\|^2\d t 
		\end{align*}
		which implies that the triplet $(\u^*,\varphi^*,\U^*)$ is an optimal triplet for the problem \eqref{control problem}. 
	\end{proof}

	 \end{document}